\documentclass [12pt,oneside, makeidx]{amsart}
\usepackage[hscale=.75, vscale=.75]{geometry}
\usepackage[utf8]{inputenc}
\usepackage{thmtools}
\usepackage{graphicx}
\usepackage{amsmath}
\usepackage{amsthm}
\usepackage{amssymb}
\usepackage[all]{xy}
\usepackage{tikz-cd}
\usepackage{verbatim}
\usepackage{mathtools}
\usepackage{floatrow}
\usepackage{subcaption}
\usepackage{enumitem}
\usepackage{url}
\usepackage{hyperref}
\usepackage{comment}
\usepackage{fourier}
\usepackage{stmaryrd}
\usepackage{scrtime}
 \usepackage[mathscr]{eucal}
 \usepackage{calligra}
 \usepackage[bbgreekl]{mathbbol}
 \usepackage{todonotes}
\usepackage{xcolor}
\definecolor{cadmiumgreen}{rgb}{0.0, 0.42, 0.24}

\newtheorem{thm}{Theorem}[section]
\newtheorem{cor}[thm]{Corollary}
\newtheorem{lem}[thm]{Lemma}

\newtheorem{prop}[thm]{Proposition}

\newtheorem{defn}[thm]{Definition}
\theoremstyle{definition}
\newtheorem{prop-defn}[thm]{Proposition and Definition}

\theoremstyle{remark}
\newtheorem{rem}[thm]{Remark}

\newtheorem*{Question*}{\bf Question}

\numberwithin{equation}{section}

\newcommand{\thistheoremname}{}
\newtheorem*{genericprop*}{\thistheoremname}
\newenvironment{namedprop*}[1]
  {\renewcommand{\thistheoremname}{#1}%
   \begin{genericprop*}}
  {\end{genericprop*}}

\newtheorem*{genericlem*}{\thistheoremname}
\newenvironment{namedlem*}[1]
  {\renewcommand{\thistheoremname}{#1}%
   \begin{genericlem*}}
  {\end{genericlem*}}
  
  \newtheorem*{genericthm*}{\thistheoremname}
\newenvironment{namedthm*}[1]
  {\renewcommand{\thistheoremname}{#1}%
   \begin{genericthm*}}
  {\end{genericthm*}}

  \newtheorem*{genericcond*}{\thistheoremname}
\newenvironment{namedcond*}[1]
  {\renewcommand{\thistheoremname}{#1}%
   \begin{genericcond*}}
  {\end{genericcond*}}

\makeindex
\usepackage[
  backend=biber,
     style=ext-alphabetic,
    sorting=nyt,
    giveninits,
    articlein=false,
    url=true, 
    doi=true,
    eprint=true
]{biblatex}

\DeclareSourcemap{
  \maps[datatype=bibtex]{
    \map[overwrite]{
      \step[fieldsource=doi, final]
      \step[fieldset=url, null]
    }  
  }
}

\DeclareFieldFormat[article, incollection, unpublished]{title}{\sl{#1}}

\newcommand{\Real}{{\mathbb R}}

\newcommand{\Hom}{{\rm Hom}}

\newcommand{\colim}{\mathrm {colim}}

\newcommand{\Fun}{\operatorname{Fun}}
\newcommand{\Mod}{\operatorname{Mod}}
\newcommand{\Yo}{\operatorname{\mathcal{Y}}}

\newcommand{\C}{{\mathcal C}}
\newcommand{\D}{{\mathcal D}}

\newcommand{\F}{{\mathcal F}}
\newcommand{\G}{{\mathcal G}}

\newcommand{\K}{{\mathcal K}}

\newcommand{\V}{{\mathcal V}}
\newcommand{\W}{{\mathcal W}}
\newcommand{\J}{{\mathcal J}}
\newcommand{\M}{{\mathcal M}}

\newcommand \id {{\rm id}}

\newcommand{\supp}{{\rm supp}}

\newcommand {\pt}{{\rm pt}}

\makeindex

\newcommand  {\R} {\mathcal R}

\DeclareRobustCommand{\SkipTocEntry}[5]{}
\DeclareMathAlphabet{\mathpzc}{OT1}{pzc}{m}{it}
\usepackage[percent]{overpic}
\usepackage{datetime}

\newcommand{\DHam }{\mathfrak{DHam}}

\renewcommand{\to}[1][]{\xrightarrow[]{#1}}

\newcommand{\isoto}[1][]{\xrightarrow[#1]%
{{\raisebox{-.6ex}[0ex][-.6ex]{$\mspace{1mu}\sim\mspace{2mu}$}}}}

\newcommand{\N}{\mathbb{N}}

\newcommand{\End}{\operatorname{End}}

\newcommand{\homst}{\operatorname{\mathcal{H}om}^\cstar}

\def \Z {\mathbb Z}
\def \H {\mathcal H}

\def \cstar {\ast}

\def \DHam {{\mathrm {DHam}}}

\newcommand{\Tam}{\mathcal{T}}
\newcommand{\Der}{\operatorname{D}}
\newcommand{\fMod}{{\operatorname{Mod^{fil}}}}
\newcommand{\Derf}{\operatorname{D^{fil}}}

\def \cstar {\varoast}

\newcommand{\larbR} { \overleftarrow{\mathbb{R}}}
\newcommand{\arbR} { \overrightarrow{\mathbb{R}}}
\newcommand{\Op}{{\operatorname{Op}}}

\makeatletter
\newcommand\boxstar{\mathrel{\boxcls@{\ast}}}
\newcommand{\boxcls@}[1]{%
  \vphantom{\Box}%
  \mathpalette\boxcls@@{#1}%
}
\newcommand{\boxcls@@}[2]{%
  \ooalign{$\m@th#1\Box$\cr
  \hidewidth\boxcls@fix{#1}\hbox{$\m@th#1#2\mkern 2.4mu$}\cr}}

\newcommand\boxcls@fix[1]{%
  \ifx#1\displaystyle
    \raise.225ex
  \else
    \ifx#1\textstyle
      \raise.225ex
    \else
      \ifx#1\scriptstyle
        \raise.180ex
      \else
        \raise.150ex
      \fi
    \fi
  \fi
}
\makeatother

\newcommand{\Ch}{{\operatorname{Ch}}}
\newcommand{\fCh}{{\operatorname{Ch^{fil}}}}
\newcommand{\Sh}{{\operatorname{Sh}}}
\newcommand{\PSh}{{\operatorname{PSh}}}

\newcommand{\op}{{\operatorname{op}}}
\renewcommand{\SS}{{\operatorname{SS}}}

\newcommand{\Fuk}{{\mathscr{F}}}
\newcommand{\WFuk}{{\mathscr{WF}}}
\newcommand{\RS}{{\operatorname{RS}}}

\newcommand{\Torus}{\mathbb T}

\title{Density of fibers for the filtered Fukaya category of $T^*N$}

\author{St{\'e}phane Guillermou}
\address{UMR CNRS 6629 du CNRS, Laboratoire de Mathématiques Jean LERAY,
2 Chemin de la Houssinière, BP 92208, F-44322 NANTES Cedex 3 France}
\email{Stephane.Guillermou@univ-nantes.fr}
\author{Claude Viterbo} 
\address{Université Paris-Saclay, CNRS UMR 8628, Laboratoire de mathématiques d’Orsay, 91405, Orsay, France.}
\email{claude.viterbo@universite-paris-saclay.fr}
\author{Bingyu Zhang} 
\address{Department of Mathematics, Kyiv School of Economics, 3 Mykoly Shpaka St, Kyiv, 02000, Ukraine.}
\email{bzhang@kse.org.ua}
 
\thanks{The first two authors were supported by COSY (ANR-21-CE40-0002). The third author was supported by the Novo Nordisk Foundation grant NNF20OC0066298 and VILLUM FONDEN, VILLUM Investigator grant 37814.}

\begin{document}

\date{\today}

\begin{abstract}
  We answer a question of Biran and Cornea about the density of iterated cones of fibers in the Fukaya category of a cotangent bundle. We prove that indeed if we take a dense set of basepoints, the iterated cones of the cotangent fibres are dense in the Filtered Fukaya category. In an appendix we prove that the space of exact Lagrangians in a symplectic manifold is never totally bounded for the spectral distance (unless it is empty). This was implicit in \cite{MCA-VH-CV} for $n=1$ and proved for cotangent bundles of negatively curved manifolds in \cite{A-B-C}. 
\end{abstract}

\maketitle

\renewcommand\contentsname{\vspace*{0pt}}
\renewcommand{\baselinestretch}{1}
\tableofcontents
\renewcommand{\baselinestretch}{1}\normalsize

\newpage
\section{Introduction}

Abouzaid proved in \cite{Abouzaid-fiber-generation} that the wrapped Fukaya category of a cotangent bundle is generated by one cotangent fiber. In the filtered case, this can not happen: the Hamiltonian perturbation of a Lagrangian is isomorphic to the original Lagrangian only when they geometrically coincide. However the filtered Fukaya category comes with a notion of interleaving distance and Paul Biran~\cite{Biran-V63}asked whether the iterated cones on the cotangent fibers generate a dense subcategory, in the sense that any Lagrangian is arbitrarily close to an iterated cone of cotangent fibers. We refer to \cite{Ambrosioni-Filtered-Fukaya} for a construction of the filtered Fukaya category and~\cite{Biran-Cornea-Zhang} for its persistence structure, which yields a notion of distance. Similar density considerations had been evoked in \cite{Fukaya-GHdistance}.

We shall use the filtered Fukaya category constructed in~\cite{Ambrosioni-Filtered-Fukaya}. Its objects are the compact Lagrangians so it does not contain the cotangent fibers. To make sense of the question of density of iterated cones of the fibers, we shall embed this category in its module category through the Yoneda embedding.  

Let us introduce some notations. Let $N$ be a connected closed manifold and $DT^*N$ its unit ball cotangent bundle (with respect to a Riemannian metric) with Liouville form $\lambda$.  We denote by $\Fuk(DT^*N)$ the filtered Fukaya category whose objects are compact (exact) Lagrangian branes, that is pairs $(L, f_L)$ where $L$ is a closed exact Lagrangian contained in $DT^*N$ and $f_L$ is a primitive of $\lambda_{\mid L}$. Morphisms from $L$ to $L'$ are given by the Floer cochains $FC^*(L,L')$. The space of Floer cochains is filtered once we have primitives $f_L, f_{L'}$ of $\lambda_{\mid L}, \lambda_{\mid L'}$: an intersection point $x\in L\cap L'$ has filtration degree $f_{L'}(x)-f_L(x)$. 

We let $\Yo\colon \Fuk(DT^*N) \to \fMod(\Fuk(DT^*N))$ be the Yoneda embedding.  Here $\fMod(\mathcal C)$ is the set of functors from $\mathcal C$ to the category of filtered chain complexes, and the Yoneda morphism is given by associating to $L$ the functor $L' \mapsto FC^*(L',L;t)$ where $FC^*(L',L;t)$ is the set of Floer chains of action greater than $t$. For simplicity we denote by $FC^*(L,L')$ the Floer chain complex with its filtration being understood, so the Yoneda embedding will send $L$ to $L' \mapsto FC^*(L',L)$.

Let $V_{(x,a)}$ be the fiber at $x$ with primitive $a$ of $\lambda_{\mid V_x}=0$. This Lagrangian brane $V_{(x,a)}$ does not belong to $\Fuk(DT^*N)$ but it makes sense to define the filtered complex $FC^*(L', V_{(x,a)})$ for any $L' \in \Fuk(DT^*N)$ and this gives a module $\V_{(x,a)} \in \fMod(\Fuk(DT^*N))$.  The action of $FC^*(L,L')$ on $FC^*(L', V_{(x,a)})$ is given by the standard triangle product
\[FC^*(L,L')\otimes FC^*(L', V_{(x,a)}) \longrightarrow FC^*(L, V_{(x,a)})\]
which respects the filtration by \cite{Ambrosioni-Filtered-Fukaya}. 

The module category $\fMod(\Fuk(DT^*N))$ inherits an interleaving distance denoted by $\gamma$.  It is also a pre-triangulated category. 

We define the subcategory $\langle \V_{(x_1,a_1)}, \ldots, \V_{(x_l,a_l)} \rangle$ of $\fMod(\Fuk(DT^*N))$ the subcategory generated by the $\V_{(x_i,a_i)}$'s as the subcategory having for objects the iterated cones on the generators.

Now we can state our main result.

\begin{thm}[Density Theorem]\label{thm:density Fukaya}
  For any closed exact Lagrangian $L \in \Fuk(DT^*N)$ and $\varepsilon >0$, there are points $(x_i)_{i\in \{1, \ldots,l\}}$, real numbers $(a_i)_{i\in \{1, \ldots,l\}}$ and $C \in \langle \V_{(x_1,a_1)}, \ldots, \V_{(x_l,a_l)} \rangle$ such that $\gamma (L,C) < \varepsilon$.
\end{thm}

The strategy of proof is to use the quantization functor constructed in~\cite{Viterbo-Sheaves} which associates a sheaf $Q(L)$ on $N\times\Real$ with any closed exact Lagrangian brane $L$.  In this way we embed $\Fuk(DT^*N)$ into $\Sh(N\times\Real)$, more precisely into the subcategory $\Sh_{DT^*N}(N\times\Real)$ (actually the Tamarkin category $\Tam_{DT^*N}(N)$ introduced later), formed by the sheaves with reduced microsupport contained in $DT^*N$. The category $\Sh_{DT^*N}(N\times\Real)$ should be understood as a model of the wrapped Fukaya category in the filtered setting. The fiber $\V_{(x,a)}$ corresponds to the sheaf $k_{\{x\} \times [a, \infty)}\in \Sh(N\times\Real)$, which does not belong to $\Sh_{DT^*N}(N\times\Real)$. We let $P'_{DT^*N}$ be a projector which takes values in $\Sh_{DT^*N}(N\times\Real)$ (it will be recalled in \S\ref{sec:strategyproof} - see~\cite{Kuo-wrappedsheaves} where it is built as a wrapping functor) and we set $\W_{(x,a)} = P'_{DT^*N}(k_{\{x\} \times [a, \infty)})$. A slight modification of~\cite[Thm. 4.4]{Zhang-Cyclic} shows that the endomorphism algebra of $\W_{(x,a)}$ computes the homology of length-filtered based loop spaces, which further clarifies the role of $\Sh_{DT^*N}(N\times\Real)$ as ``wrapped filtered Fukaya category''. The sheaf category also inherits an interleaving distance that we denote by $\gamma^s $. 

Now in the category $\Sh_{DT^*N}(N\times\Real)$ we can prove a more general result: 
\begin{thm}[Density Theorem]\label{thm:density sheaf}
  Let $\F \in \Sh_{DT^*N}(N\times\Real)$ such that for all $x\in N$ the sheaf $\F \otimes k_{\{x\} \times\Real}$ is a $\gamma^s$-limit of constructible sheaves on $\Real$. Then, for any $\varepsilon >0$, there are points $(x_i)_{i\in \{1, \ldots,l\}}$, real numbers $(a_i)_{i\in \{1, \ldots,l\}}$ and $C \in \langle \W_{(x_1,a_1)}, \ldots, \W_{(x_l,a_l)} \rangle$ in ${\Sh_{DT^*N}(N\times\Real)}$ such that $\gamma^s (\F,C) < \varepsilon$.
\end{thm}

The idea is to use a \v Cech resolution of the constant sheaf $k_N \simeq C_0 \to C_1 \to \cdots \to C_m$, where $C_i = \bigoplus_{J \subset I, |J| = i+1} k_{U_J^{cl}}$ with $U_J^{cl} = \bigcap_{j\in J} \overline{U_j}$ and the $U_j$'s are $\varepsilon$-small balls covering $N$. Then $\F = \F \otimes k_{N\times\Real}$ is written as an iterated cone on the $\F \otimes k_{U_J^{cl} \times\Real}$. Now $\F \otimes k_{U_J^{cl} \times\Real}$ can be approximated by $\F \otimes k_{\{x\} \times\Real}$ for some $x \in U_J^{cl}$ and $\F \otimes k_{\{x\} \times\Real} $ (a sheaf on a line) can be written as (in general can be approximated by) an iterated cone on the $\W_{(x,a)}$'s by constructibility. From the argument, we know that the points $(x_i)_{i\in \{1, \ldots,l\}}$ can be taken from an a priori given dense subset of the base manifold $N$. Also, the proof may be viewed as a filtered analogue of the sectorial descent of~\cite{G-P-S-3}, formulated in the language of sheaves.

\begin{rem}
In fact, as in \cite{A-B-C} we prove something stronger than density, we prove what they call
{\it approximability}: For any $\varepsilon>0$ we can find objects $X_1,...,X_N$ such that any object in the category is at distance at most $ \varepsilon$ from an iterated cone of the $X_k$. Even better if we consider the set of all direct sums  obtained from the $X_i$, we need only $n+1$-iterated cones (where $n$ is the dimension of the base manifold). A related notion of complexity called the interleaving Rouquier dimension will be discussed in Appendix~\ref{sec:InterleavingRouquierdimension}.  
\end{rem}
Once we have the result for sheaves, we can come back to $\fMod(\Fuk(DT^*N))$ using the quantization functor $Q$ since sheaf quantization of closed exact Lagrangian satisfies the assumption of the Density Theorem, due to constructibility of $Q(L) \otimes k_{\{x\} \times\Real} \simeq FC^*(V_{(x,0)},L)$ as sheaves on $\Real$. In fact, the sheaf result is a little stronger in the sense that it may apply to sheaves associated with immersed Lagrangians or certain $C^0$-Lagrangians. One should also notice that the Fukaya category may depend on perturbation data; however, the sheaf category does not. Then it tells that the sheaf distance bound gives a uniform bound of Fukaya category distance for all perturbation data.

\medskip

To implement this plan of proof we need to be more precise on the quantization functor $Q$. In~\cite{Viterbo-Sheaves} the quantization functor is defined on the Donaldson-Fukaya category of $T^*N$ with value in the (classical) derived category of sheaves. However the Donaldson-Fukaya category is not triangulated and we cannot state our result in this framework. For this reason we will first enhance $Q$ to a functor defined on the Fukaya category in a higher coherent way, and then we can extent $Q$ to the category of modules over Fukaya category.  Recall that the starting point in~\cite{Viterbo-Sheaves} is to first define a presheaf whose sections on some open set $U \times (-\infty, a)$ is $\lim_f FH(\Gamma_{df}, L)$ where $f$ is a function running over the smooth functions greater than the characteristic function of $U$, rescaled by $a$.  To turn this into a functor we need some functoriality of homotopy limits. This may be true in the framework of $A_\infty$-categories but we lack references. So we will turn our categories into $\infty$-categories, for which the appropriate results are available in the literature. In the sequel and for simplicity, all coefficients will be in $\mathbb Z/2 \mathbb Z$. 

\section{Acknowledgements}
The second and third authors first heard about the question of the density of cones on the fibers in the cotangent bundle (and in fact of other Lagrangians in other symplectic manifolds) from a talk by Paul Biran in June 2024 (\cite{Biran-V63}. The approach he mentioned (still incomplete at the time) is radically different from the one given here and is now finalized in \cite{A-B-C}. We believe both proofs have their own merit.  We warmly thank Paul Biran for sharing these ideas on this occasion as well as Giovanni Ambrosioni and Octav Cornea for several useful discussions.
The third author thanks Tatsuki Kuwagaki, Adrian Petr, and Vivek Shende for sharing their progress on their independent work regarding similar considerations. He thanks them as well as Laurent C\^ot\'e and Zhen Gao for helpful discussions. The second author thanks Marie-Claude Arnaud and Vincent Humilière for permission to use the results and illustrations of \cite{MCA-VH-CV} in Appendix B.

\section{Categories}

We first introduce some notations. For a dg-category $\C$, its dg nerve is a simplicial set $N_{dg}(\C)$ defined as follows: $N_{dg}(\C)$  is the simplicial set such that an $n$-simplex $f\in N_{dg}(\C)_n$ consists of the following data $f=(\{X_i\}_{0\leq i \leq n},\{f_I\})$, where the $X_i$ are objects in $\C$; and for each ordered subset $I = \{  i_0 > i_{1} > \cdots > i_ k \}  \subseteq [n]$ having at least two elements, $f_I \colon X_{i_k}\rightarrow X_{i_0}$ is a degree $k-1$ map\footnote{Remember that morphisms in a dg-category are graded} satisfying the coherent cocycle condition 
\begin{equation}\label{dg-nerve cocycle}
  \partial f_{I} = \sum _{a=1}^{k-1} (-1)^{a} ( f_{ \{  i_0 > i_1 > \cdots > i_ a \}  } \circ f_{ \{  i_ a > \cdots > i_ k \}  } - f_{I \setminus \{  i_ a \}  } ).
\end{equation}
The structural maps are defined as follows: a nondecreasing function $\alpha \colon [m]\rightarrow [n]$ induces $\alpha ^{\ast } \colon N_{dg}(\C)_m \rightarrow N_{dg}(\C)_n$ defined by  
\[\alpha^*(( \{  X_ i \} _{0 \leq i \leq m}, \{  f_{I} \}  )) = ( \{  X_{ \alpha (j) } \} _{ 0 \leq j \leq n}, \{  g_{J} \}  ),\] where  
\[g_{J} = \begin{cases}  f_{ \alpha (J) } &  \textnormal{ if } \alpha |_{J} \textnormal{ is injective } \\  \operatorname{id}_{ X_ i } &  \textnormal{ if } J = \{  j_{0} > j_{1} \}  \textnormal{ with } \alpha (j_{0}) = i = \alpha (j_{1}) \\  0 &  \textnormal{ otherwise. } \end{cases}\]
The dg nerve construction can be generalized to strict unital $A_\infty$-categories in a similar way, adding higher differential in the cocycle condition \eqref{dg-nerve cocycle}. We will denote $N_{A_\infty}(\C)$ this $A_\infty$-nerve (which is an $\infty$-category) and refer to \cite{Ainfinity_Nerve} for more details.

We consider coefficients in some commutative ring $k$ (we take $k=\mathbb Z/2 \mathbb Z$ when dealing with Fukaya categories).  We denote by $\Ch^\infty(k)$ the dg-category of cochain complexes on $k$ and by $\Ch^\infty(k) = N_{dg}(\Ch(k))$ the $\infty$-category obtained via the dg nerve functor\footnote{Remember that an $\infty$-category is a simplicial set satisfying the weak Kan extension: every inner horn has a filler.}.

We let $\Der(k) = \Ch^\infty(k)[W^{-1}]$ be its localization along the class $W$ of edges  given by quasi-isomorphisms of complexes. A model for $\Der(k) $ is $N_{dg}(\Der^{dg}(k))$ the dg nerve of the dg-derived category of $k$.

For a manifold $N$ we set $\Sh(N) = \Sh(N, \Der(k))$ as the $\infty$-category of (homotopy)-sheaves valued in $\Der(k)$ (see \cite[Definition 1.1.2.1]{SAG}). In the case of manifold, we have that $\Sh(N)$ is equivalent to the unbounded $\infty$-derived category of Abelian sheaves $\Der(\Sh(N,\Mod_k))$ (see for example \cite[Proposition 7.1]{Six-Functor-Scholze}).

\addtocontents{toc}{\SkipTocEntry}
\subsection*{Filtered categories}
A convenient way to deal with the filtration on the categories is to introduce the category of filtered complexes.  We let $\larbR$ be the $\infty$-category associated with the poset $(\Real, \geq)$ via the nerve functor.  We set $\fCh(k) = \Fun(\larbR,\Ch^\infty(k))$ and $\Derf(k) = \Fun(\larbR,\Der(k))$ where $\Fun$ will always denote $\infty$-functors. We have $\Derf(k) \simeq \fCh(k)[W_{\larbR}^{-1}]$, where $W_{\larbR}$ is the set of filtration-preserving quasi-isomorphisms, i.e., for each $t$, $f_t\colon M_t\rightarrow N_t$ is a quasi-isomorphism.

Let $\Real_{\leq}$  be the topological space $\Real$ whose open subsets are the intervals $(-\infty,a)$.  This is a special case of the $\gamma$-topology $V_\gamma$ for a cone $\gamma$ in some vector space $V$ introduced in~\cite{K-S}.  We remark that the poset $(\Real, {\geq})$ is fully faithfully embedded  (as a category) into the category $\Op(\Real_{\leq})^\op\cong(\Real\cup\{\pm \infty\}, \geq)$.  Hence there are restriction functors
\begin{equation*}
\begin{split}
\PSh(\Real_{\leq},\Ch^\infty(k))\rightarrow \fCh(k) & ,\quad \PSh(\Real_{\leq}, \Der(k)) \rightarrow \Derf(k),\\ 
\F  \mapsto M(\F)&=[t\mapsto  M(\F)_t=\F((-\infty,t))].
    \end{split}
\end{equation*}
out of the category of presheaves on $\Real_{\leq}$ valued respectively in $\Ch^\infty(k)$, $\Der(k)$. It is explained in \cite[Section 4.2]{APT-Kuwagaki-Zhang}\footnote{In \cite{APT-Kuwagaki-Zhang}, the authors use the poset $(\Real, {\leq})$, which causes a difference on sign and cohomology degree shifting.} that the functor $\PSh(\Real_{\leq}, \Der(k)) \rightarrow \Derf(k)$ restricts to a fully faithful functor $\Sh(\Real_{\leq})\hookrightarrow \Derf(k)$ whose image satisfies the semi-continuous condition for persistence modules, which is also a monoidal functor respects certain monoidal structures. Taking the left adjoint and sheafification, we obtain functors
\begin{equation}
  \label{eq:filt-cplx-sheaves}
  \fCh(k)  \to \Derf(k) \to \Sh(\Real_{\leq}) \hookrightarrow  \Sh(\Real).
\end{equation}

We recall the convenient equivalence of sheaf categories, in the $\infty$-category setting, $\Sh(N\times \Real) \simeq \Sh(N, \Sh(\Real))$ (see~\cite[Cor. 2.24 and Prop. 2.30]{Volpe-sixoperations}).  In particular, using~\eqref{eq:filt-cplx-sheaves} and sheafifying on $N$, we obtain functors
\begin{equation}
  \label{eq:pshN-shNR}
  \Fun(\Op^\infty(N)^\op, \Derf(k)) \to  \Sh(N,\Derf(k)) \to  \Sh(N \times\Real) .
\end{equation}
where $\Op^\infty(N)$ is the nerve of the poset $\Op(N)$ of open subsets in $N$.

\addtocontents{toc}{\SkipTocEntry}
\subsection*{Categories related to the microsupport}
For a manifold $M$ and $\F\in \Sh(M)$ we denote by $\SS(\F) \subset T^*M$ the microsupport of $\F$ defined by Kashiwara-Schapira~\cite{K-S}.  For a conic subset $W \subset T^*M$ we let $\Sh_W(M)$ be the subcategory of $\Sh(M)$ formed by the sheaves $\F$ with $\SS(\F) \subset 0_M \cup W$.  In the case $M = N \times \Real$ we set for short $\{\tau \geq 0\} = \{(x,p,t,\tau) \in T^* (N\times \Real) \mid  \tau \geq 0 \}$ and define $\{\tau \,R\, 0\}$ for $R\in \{\leq ,>,<\}$ in the same way. The microlocal cut-off lemma of Kashiwara-Schapira (see~\cite{K-S}, Prop.~5.2.3(i), together with Prop.~3.5.3(iii)) shows that the essential image\footnote{This is the subcategory with objects equivalent to objects in the image. See \cite[Def 4.6.2.12]{kerodon} for the $\infty$-category case.} of $\Sh(\Real_{\leq}) \hookrightarrow  \Sh(\Real)$ is identified with $\Sh_{\tau \geq 0}( \Real)$.

The Tamarkin category $\Tam(N)$ is defined as the left orthogonal of $\Sh_{\tau \leq 0}(N\times \Real)$, i.e. the full subcategory spanned by sheaves $\F$ such that $\Hom(\F,\G)=0$ for all $\G \in \Sh_{\tau \leq 0}$. We have  $\Tam(N)\hookrightarrow\Sh_{\tau\geq 0}(N\times \Real)$, and a sufficient (but not necessary) characterization is $F|_{N\times (-\infty,A)}=0$ for some $A\in \Real$.
For $W \subset T^*N$ we also let $\Tam_W(N)$ be the subcategory of $\Tam(N)$ formed by the $\F$ with $\SS(\F) \subset 0_{N\times\Real} \cup \rho^{-1}(W)$, with
\[
\rho\colon T^*N \times \{\tau>0\} \to T^*N, \qquad  (x,p,t,\tau) \mapsto (x, p/\tau) .
\] 
We denote by $\RS(\F)\coloneqq \rho(\SS(\F))$ the reduced microsupport.

\addtocontents{toc}{\SkipTocEntry}
\subsection*{Interleaving distances}

Both $\Fuk(T^*N)$ and $\Tam(N)$ come with``translation functors'' $T_a$, $a\in \Real$,  inducing an interleaving distance on the set of objects.  A ``translation functor'' for a category $C$ is a family of functors $T_a \colon C \to C$ for all $a\in \Real$ and morphisms of functors $\tau_b^a\colon T_a \to T_b$ for $a\leq b$ satisfying some natural compatibilities (higher coherently). A good way to express these compatibilities is to ask that the data of $T$ and $\tau$ gives a monoidal functor $(T,\tau)\colon (\Real, \leq) \to \End(C)$, where $(\Real, \leq)$ is seen as a category through the order with a monoidal structure given by the addition (see~\cite[Def. 1.3.4]{Petit-Schapira} or~\cite[Def. 2.15]{Biran-Cornea-Zhang}).  This condition implies in particular that $\tau_b^a \simeq T_a(\tau_{b-a}^0)$ and we will only consider the family $\tau_c = \tau_c^0$, $c\geq 0$.
For $X, Y \in C$ we then set\footnote{In the case where $C$ is an $\infty$-category, equality of morphisms is replaced by homotopy of 1-morphisms.}
\begin{align*}
\gamma(X,Y) &=  \inf \Big\{a+b \mid a,b \geq 0,\;  \exists u\colon X \to T_a(Y), v\colon Y \to T_b(X), \\
&  T_a(v) \circ u \colon X \to T_{a+b}(X), \, T_b(u)\circ v \colon Y \to T_{a+b}(Y), \text{ are the maps $\tau_{a+b}$}  \Big\}
\end{align*}
Then $\gamma$ is a pseudo-distance\footnote{Several variations are possible, for example we could ask that $a=b$.}on the set of (isomorphisms classes of) objects of $C$ (it is not a distance since we may have $\gamma(X,Y) = 0$ for non-isomorphic $X,Y$).  

If $C$ is an $\infty$-category, we replace $(\Real, \leq)$ by $\arbR$, the $\infty$-category associated with the poset $(\Real, \leq)$ via the nerve functor as before (we refer to the higher coherent version~\cite{APT-Kuwagaki-Zhang}).

\begin{rem}\label{rem:preserve distance}
(1)  Let $C, C'$ be two $\infty$-categories endowed with translation functors $(T,\tau)\colon \arbR \to \End(C)$, $(T',\tau')\colon \arbR \to \End(C')$ and let $\gamma, \gamma'$ be the associated pseudo-distances.  Let $F\colon C \to C'$ be a functor commuting with the $\arbR$-actions.  Then $F$ is $1$-Lipschitz, that is, $\gamma'(F(X), F(Y)) \leq \gamma(X,Y)$.  If, moreover, $F$ is fully faithful, then $F$ preserves the distances.

\medskip

(2) When the categories are $k$-linear, we will assume the related functors are $k$-linear as well. We will also consider the category of filtered modules of $C$ for $k$-linear $C$, $\fMod(C) = \Fun(C^{op}, \fCh(k) )$, which admits a translation functor induced by
$\fCh(k) $. If $C$ is endowed with translation functors $(T,\tau)$, then $\fMod(C)$ inherits another translation functor.  More precisely, for $X^* \in \fMod(C)$ we define $T_c^*(X^*)$ by $T_c^*(X^*)(X) = X^*(T_{-c}(X))$, where we introduce the minus sign so that the morphism of functors $\tau\colon \id \to T_c$, $c\geq 0$, induces a morphism $\tau^* \colon \id \to T_c^*$, $c \geq 0$. Then $(T^*,\tau^*) \colon \arbR \to \End(\fMod(C))$ is a monoidal functor. 

In the article, we assume that the category $C$ satisfies that the two translation functors on $\fMod(C)$ defined above are equivalent. This is the case for the filtered Fukaya category of exact symplectic manifolds and exact Lagrangian branes (see \cite[Remark 3.8]{Biran-Cornea-Zhang}).

We then have pseudo-distances $\gamma$ on $C$ and $\gamma^*$ on $\Mod(C)$.  The Yoneda functor $\mathcal{Y}_C \colon C \to \Mod(C)$ commutes with $T_c, T^*_c$ and sends $\tau$ to $\tau^*$ and is fully faithful. By~(1) it follows that $\gamma^*(\mathcal{Y}_C(X), \mathcal{Y}_C(Y)) = \gamma(X,Y)$ for any $X,Y \in C$.

\medskip

(3) In the situation of (2), we define $\Derf(C) = \Fun(C^{op}, \Derf(k) )$ and call it the filtered derived category of $C$. The filtered derived category is identified with the localization of $\fMod(C)$ with respect to filtration-preserving quasi-isomorphisms. Therefore, it induces the translation functor from $\fMod(C)$ and the localization functor $\fMod(C)\rightarrow \Derf(C)$ induces an isometry because the interleaving distance in $\fCh(k)$ does not change under filtration-preserving quasi-isomorphisms.

\end{rem}

\section{Quantization}
\label{sec:quantization}

In~\cite{Viterbo-Sheaves} the second-named author associates a sheaf with any closed exact Lagrangian. This gives a functor from the Donaldson-Fukaya category of closed exact Lagrangian to the category of sheaves.  In this section we see that this functor can be enhanced to a functor from a filtered Fukaya category to the category of sheaves.  The construction we give here is actually the same as in~\cite{Viterbo-Sheaves} but we are careful to define the functors at the level of the $A_\infty$-category or $\infty$-category and not only on the associated homotopy category.

In this article, we use the construction of filtered Fukaya categories from~\cite{ Ambrosioni-Filtered-Fukaya}\footnote{To be compatible with the filtration convention of~\cite{Viterbo-Sheaves}, our filtration convention is opposite to loc. cit., i.e., $FC^*(L',L;t)$ is the set of Floer chains of action greater than $t$.}. Even though the author constructs filtered Fukaya categories for certain non-exact Lagrangian branes of general symplectic manifolds, here we focus on closed exact Lagrangian of the cotangent bundle $T^*N$, and denote it by $\Fuk=\Fuk(T^*N)$. We also remark that the following construction is basically independent of the definition of the filtered Fukaya categories provided it is strict unital (for example, such a construction based on categorical localization approach will appear in ~\cite{Kuwagaki-Petr-Shende-CompleteFukaya}). The coefficient ring could be an arbitrary commutative ring if moduli spaces are well-oriented. 
\begin{rem} Here, we did not emphasize the role of perturbation data in the definition of Fukaya categories and the following construction of continuation morphisms, as well as sheaf quantization functors. 

As explained in~\cite{A-B-C}, we may also organize Fukaya categories for different perturbation data as a system of filtered $A_\infty$-categories with comparison functors. In principle, those comparison functors should be compatible with continuation morphisms and sheaf quantization functors we constructed below, but we will not discuss that compatibility. We will see later that those constructions work uniformly for (but depend on) different perturbation data, and the final Fukaya density theorem does not involve perturbation data, since our sheafy result is perturbation-independent. Therefore, we do not need to pay much attention to perturbation data here.
\end{rem}

\subsection{Continuation morphisms}
We take $C_\leq = (C^\infty(N),\leq)$ to be the poset of smooth functions, and take $k[C_\leq]$ the freely generated $k$-linear category by setting $\Hom(f,g)=ke_{fg}$ if and only if $f\leq g$, and treat it as an $A_\infty$-category with trivial $\mu^k$ for $k\not=2$, and $\mu^2(e_{gh},e_{fg})=e_{fh}$if and only if $f\leq g\leq h$. We define a filtration on $\Hom(f,g)=k$ by setting $\Hom^{\leq t}(f,g)=ke_{fg}$ if and only if $t\geq 0$.

We define a \textit{coherent choice of monotone continuation morphisms} as a strict unital $(\Fuk,k[C_\leq])$ ($A_\infty$-)bimodule $B$ with $B(L,f)=FC_*(L,\Gamma_{df})\in \fCh(k)$ together with bimodule maps (see \cite{Lyubashenko2008})\footnote{In principle the map should be \[\Fuk(L_k,L_{k-1})\otimes \cdots\otimes\Fuk(L_1,L_{0})\otimes B(L_0,f_0)\otimes \Hom(f_1,f_0)\otimes \cdots\otimes \Hom(f_n,f_{n-1})\rightarrow B(L_k,f_n)[1-k-n]\] but all the $\Hom(f,g)$ are $ke_{f,g}$ or $0$ so we omit them from the notation. }

\[\mu^{k|1|n}\colon \Fuk(L_k,L_{k-1})\otimes \cdots\otimes\Fuk(L_1,L_{0})\otimes B(L_0,f_0)\rightarrow B(L_k,f_n)[1-k-n] \]
that satisfies the usual bimodule Stasheff identity and preserves filtration.

Here, we explain the existence of a coherent choice of monotone continuation morphisms.

\subsubsection*{Review of construction of the right $k[C_\leq]$-module}
The right $k[C_\leq]$-module structure is constructed based on family Floer theory as explained in~\cite[Prop. 4.6]{Viterbo-Sheaves}. 
We shall now review its construction. 
For a finite collection of functions $(f_0\leq \cdots \leq f_n)$, we can construct a  $\Delta^n$-family (here $\Delta^n$ is the topological standard $n$-simplex) of $\sigma(\lambda_1,\ldots,\lambda_n)=\sum_{j=0}^{n}(\lambda_{j+1}-\lambda_{j})f_{n-j}$ where $\lambda_{0}=0\leq \lambda_1\leq \cdots \leq \lambda_n\leq  \lambda_{n+1}=1$ (notice that the set of $(\lambda_1,\ldots,\lambda_n)$ defines a topological $n$-simplex),  and then consider a parameterized family of trajectories of the mixed Floer equation 
\[\partial_su_{\sigma(\lambda(s))}(s,t)+J\partial_t u_{\sigma(\lambda(s))}(s,t)=0,\quad \partial_s \lambda(s)=\tau\zeta_g(\lambda(s)),\]
where $\zeta_g(\lambda)=\sum_{j=1}^n \lambda_j(1-\lambda_j) \partial_{\lambda_j}$, subject to the boundary condition $u_{\lambda(0)} \in L$, $u_{\lambda(1)} \in \Gamma_{d\sigma(\lambda)}$, $\lim_{s\rightarrow \pm \infty} u_{\lambda(s)}(s,t)\in L\cap L_{\sigma(\lambda)}$, $\lim_{s\rightarrow - \infty} \lambda(s)=(0,\ldots,0)$, $\lim_{s\rightarrow   \infty} \lambda(s)=(1,\ldots,1)$.
Note that $\tau$ is a small parameter introduced because one needs ``slow homotopies'' and in the above chain map - which in principle depends on $\tau$ - we take the limit as $\tau$ goes to $0$. 
To solve the equation, we transform the moving boundary condition problem to a usual Floer equation by introducing a Hamiltonian perturbation $H$ as constructed in~\cite[Prop. 3.3, 4.6]{Viterbo-Sheaves}. 

For $y_0\in FC_*(L,\Gamma_{df_0})$ and $y_1\in FC_*(L,\Gamma_{df_n})$, we denote by $\M(y_0,y_1)$ the moduli space of rigid solutions of the corresponding Floer equation, and then we define the right $k[C_\leq]$ module map 
\begin{align*}
B(L_0,f_0)=FC_*(L,\Gamma_{df_0}) &\rightarrow B(L_0,f_n)=FC_*(L,\Gamma_{df_n}), \\
y_0 &\mapsto \sum_{y_1} \# \M(y_0,y_1) y_1,    
\end{align*}
where the Stasheff identity corresponds exactly to Equation~\eqref{dg-nerve cocycle}, which means that all those right module maps associated with $(f_0\leq \cdots \leq f_n)$ and its subsets form a $n$-simplex in $\fCh(k)$.

\subsubsection*{Construction of the $(\Fuk,k[C_\leq])$-bimodule}
Now, we show the compatibility with left multiplication of $\Fuk$. The construction follows from the idea of \cite[Sec. 5]{G-P-S-2}, but simpler. 

\begin{prop}\label{Prop-4.1}There exists a coherent choice of monotone continuation morphisms for $\Fuk$, whose right module restriction is compatible with family Floer construction in ~\cite[Prop. 4.6]{Viterbo-Sheaves}. 
\end{prop}
\begin{proof}

We start from the case that all Lagrangians $(L_k,\ldots, L_0)$ intersect transversely.

We consider moduli spaces of holomorphic strips with marked points: let $\R_{k;n}^B$ be the compactified moduli space of strips $\Real \times [0,1]$ with $k$ marked points $z_1,...,z_k$ on $\Real \times \{0\}$ boundary marked points and $n$-tuple of reals $a_1\geq \cdots \geq a_n$ defined up to translation and view them as marked points on $\Real \times \{1\}$, see Figure~\ref{fig: Lagrangian labeling of moduli spaces}. We refer to \cite[Sec. 5.3]{G-P-S-2} for details of a similar compactification. We set the universal curve $\C_{k;n}^B=\{(D,z)\mid D\in \R_{k;n}^B, z\in D\}$ together with sections $z_i$, $i=1,\ldots, k$ and $a_j$, $j=1,\ldots,n$. 

As described in \cite[Sec. 4.1]{G-P-S-2}, $\R_{0;n}^B$ describe a space of Morse flow lines on $\Delta^n$, we could regard a configuration $a_1\geq \cdots \geq a_n$ as a universal map from the $\Real \times \{1\}$ part of $\C_{k;n}^B$ to $\Delta^n$ by associating $a_j$ to the edge $(j-1,j)$ and associating the interval $(a_{j+1},a_j)$ to the vertex $j$ ($a_0=\infty, a_{n+1}=-\infty$). 

The Lagrangian labeling on $\Real \times \{0\}$ associates $\overline{z_{i-1}z_i}$ to $L_i$ for $i=0,\ldots, k+1$ ($z_{-1}=-\infty$, $z_{k+1}=+\infty$). The identification of $\Real \times \{1\}$ with $\Delta^n$ compose with the $n$-dimension family of $\sigma\colon \Delta^n\rightarrow C^\infty(N)$ give the Lagrangian labeling on $\Real \times \{1\}$ that associates $\overline{a_{j+1}a_j}$ to $\Gamma_{df_{j}}$. Those data form the domain data of our counting. 

Conformally, we may regard $D\in \R_{k;n}^B$ as a disk with $k+2$ marked points and $n$ marked points located in the interior of the arc $\overline{z_{k+1}z_0}$ where $z_0=-\infty$ and $z_{k+1}=\infty$. 

\begin{figure}[htbp]
    \centering
\begin{tikzpicture}[scale=1.2]

  \def\L{3.6}
  \def\H{1}  
  \def\D{1.2}

  \draw[thick] (-\L,0) -- (\L,0);
  \draw[thick] (-\L,\H) -- (\L,\H);

  \path[fill=gray!20,draw=none]
    (-\L, 0.5*\H) rectangle (\L,\H);

  \node at (0,0.7*\H) {$H$};

  \fill (-1.4*\D,0) circle (2pt) node[above] {$z_1$};
  \fill (-0.7*\D,0)   circle (2pt) node[above] {$z_2$};
   \fill (0.2*\D,0)   circle (2pt) ;
  \fill (1.5*\D,0) circle (2pt) node[above] {$z_k$};

  \fill (-1.2*\D,\H) circle (2pt) node[below] {$a_{n}$};
  \fill (-0.4*\D,\H) circle (2pt) node[below] {$a_{n-1}$};
  \fill (0.8*\D,\H) circle (2pt) ;
    \fill (1.6*\D,\H) circle (2pt) node[below] {$a_{1}$};

  \node[above] at (-2*\D,\H)   {$\Gamma_{df_n}$};
  \node[above] at (-0.8*\D,\H)   {$\Gamma_{df_{n-1}}$};
  \node[above] at (0.2*\D,\H)   {$\Gamma_{df_{n-2}}$};
  \node[above] at (2.4*\D,\H)   {$\Gamma_{df_0}$};

  \node[below] at (-2*\D,0)   {$L_0$};
  \node[below] at (-1*\D,0)   {$L_1$};
  \node[below] at (-0.2*\D,0)   {$L_2$};
    \node[below] at (0.8*\D,0)   {$\cdots$};
  \node[below] at (2.4*\D,0)   {$L_k$};

  \node[left] at (-1.2*\L,0.5*\H) {$s=-\infty$};
  \node[right] at (1.2*\L,0.5*\H) {$s=\infty$};

\end{tikzpicture} 
    \caption{Lagrangian labeling and Floer data for the bimodule $B$}
    \label{fig: Lagrangian labeling of moduli spaces}
\end{figure}

Next, we give the Floer datum: We pick standard strip-like ends $\epsilon_i$ near $z_i$ and Hamiltonian $1$-form $K$, and almost complex structures $J\colon  \C_{k;n}^B \rightarrow \J(T^*N)$.
We take a Hamiltonian family $H\colon  \C_{k;n}^B\times T^*N\rightarrow \Real$ that gives the Hamiltonian perturbation introduced by ~\cite[Prop. 3.3]{Viterbo-Sheaves} near $\Real \times \{1\}$ and trivial otherwise. Moreover, the choices of Floer data should be compatible  as described in \cite[Sec. 5.4]{G-P-S-2}.

We also assume that we choose a sequence of functions $f_0\leq \cdots\leq f_n$ such that $f_j-f_i$ are Morse for $i\neq j$ and, in addition, that $\Gamma_{f_0} \pitchfork L_0$ and $\Gamma_{f_n}\pitchfork L_k$, then we take $x_i\in L_{i-1}\cap L_i$, $i=1,\ldots,n$, and $y_0\in L_{0}\cap \Gamma_{df_0}$, $y_1\in L_{k}\cap \Gamma_{df_n}$. We consider the moduli space $\M(x_k,\ldots,x_1,y_0;y_1)$ of solutions of 
\[(du-X_{H}(u)\otimes K)^{0,1}_{J}=0\]
with the given boundary condition. Since all Lagrangians are closed exact and intersect transversely, then compactness and transversality can be achieved by regular Floer datum. Then we define the bimodule map by the rigid counting
\[\begin{split}
   \mu^{k|1|n}_B\colon \Fuk(L_k,L_{k-1})\otimes \cdots \otimes \Fuk(L_1,L_{0})\otimes B(L_0,f_0)&\rightarrow  B(L_k,f_n)[1-k-n] \\
   x_k\otimes \cdots\otimes x_1\otimes y_0\hspace*{2cm}&\mapsto  \# \M(x_k,\ldots,x_1,y_0;y_1) y_1.
\end{split}\]

For the Stasheff identity, we see that the Hamiltonian $H$ is trivial far away from $\Real \times \{1\}$, so homomorphic polygons split as the usual $A_\infty$ relation near $\Real \times \{0\}$; and near $\Real \times \{1\}$, by our choice of $H$ and $J$, the splitting of holomorphic strips at $s=\pm \infty$ are the same as ~\cite[Prop. 4.6]{Viterbo-Sheaves} or ~\cite[Lemma 4.33]{G-P-S-2} described in a simplicial way.

To finish the existence, we discuss the non-transverse intersection case. In this case, we really need to follow \cite{Ambrosioni-Filtered-Fukaya} to count holomorphic clusters rather than holomorphic disks. In case all intersections are transverse, the counting of holomorphic clusters reduces to counting of disks (or strips), so our previous discussion works directly in this case (and actually gives the same bimodule maps). 

The most important thing here is the existence of regular Floer data such that all operations preserve filtration. This is the main contribution of \cite{Ambrosioni-Filtered-Fukaya} in case there is no Hamiltonian term $H$. In our case  the Hamiltonian term $H$ is associated to a ``monotonic” family of exact graphs, hence it preserves the filtration. The existence of regular $\epsilon_i, K,J$ then follows from \cite{Ambrosioni-Filtered-Fukaya}. 

It follows that we can also construct the bimodule maps in the non-transverse intersection case (but we need to be more careful when defining the moduli space of the domain) by replacing holomorphic strips with holomorphic clusters.

 To check compatibility of this bimodule structure in the case $k=0$ with that from \cite[Prop. 4.6]{Viterbo-Sheaves}
 we use the fact that $\R_{0;n}^B$ is identified with a space of Morse flow lines on $\Delta^n$ we described before. Note that the pseudo-gradient in \cite[Remark 4.1]{G-P-S-2} is different from the one in \cite{Viterbo-Sheaves}, but a standard Floer homotopy argument shows that our perturbed counting is equal to the one from \cite{Viterbo-Sheaves}. 
\end{proof}

\subsection{Quantization functor} For a coherent choice of monotone continuation maps $B$ (given by Proposition \ref{Prop-4.1}) we may use \cite[Proposition 5.3, Page 569]{Lyubashenko2008}
to  define a strict unital $A_\infty$-functor $k[C_\leq] \rightarrow \Mod-\Fuk$. In fact, the essential image consists of Yoneda modules $\mathcal{Y}_{\Gamma_{df}}$, so the $A_\infty$-functor factors through an $A_\infty$-functor $b\colon k[C_\leq] \rightarrow \Fuk$. The morphism $e_{fg}\in \Hom_{k[C_\leq]}(f, g)$ is mapped to a degree $0$ and non-negative action element $FC^{0}(\Gamma_{df},\Gamma_{dg};0)$ (in particular, $f=f$ is mapped to the unit of $\Gamma_{df}$ in $\Fuk$), so the functor $b\colon k[C_\leq] \rightarrow \Fuk$ preserves the filtration.

It is clear that there exists a functor $C_\leq\rightarrow k[C_\leq]$ (between 1-categories): it identifies objects on boths sides and sends $*\in \Hom_{C_\leq}(f,g)$ to $e_{fg}$. 

By using the nerve construction of 1-category or $A_\infty$-categories 
we turn these categories into 
$\infty$-categories.
In particular we set \[C_\leq^\infty\coloneqq N(C_\leq),\quad
\Fuk^\infty \coloneqq N_{A_\infty}(\Fuk).\]Notice that $N(k[C_\leq])=N_{A_\infty}(k[C_\leq])$.
We then deduce from $b$ a functor
\newcommand{\Gr}{\mathrm {Gr}}
\begin{equation}
  \label{eq:functions-to-Fukaya}
\Gr\colon  C_\leq^\infty \to \Fuk^\infty,\quad f\mapsto \Gamma_{df}.
\end{equation} 
 
Let us first recall the idea of the construction in~\cite{Viterbo-Sheaves}.  Let $L$ be a closed exact Lagrangian. We define a sheaf $F_L$ on $N\times \Real$ by its sections on $U\times (-\infty, c)$ for any $c$; the space of these sections is $\Hom(k_{U\times (-\infty, c)}, F_L)$. There is no object in the Fukaya category of closed exact Lagrangians corresponding to $k_{U\times (-\infty, c)}$. However we have a fiber sequence \[k_{U\times (-\infty, c)} \to k_{N\times \Real} \to k_{(N\times \Real) \setminus (U\times (-\infty, c))} \to[+1]\] and $k_{(N\times \Real) \setminus (U\times (-\infty, c))} \simeq \colim_{f> \chi} k_{\{t\geq f\}}$ where the colimit runs over smooth bounded functions $f$ and $\chi = c$ on $U$, $\chi=-\infty$ outside $U$\footnote{The sheaf $k_{U\times (-\infty, c)}$ corresponds to, in the infinitesimal wrapped Fukaya category,  the external conormal $T_{\partial U,+}^*N$ of $\partial U$, when $U$ has a smooth boundary. The sheaf result indicates that $T_{\partial U,+}^*N$ should be the colimit $\colim_{f> \chi} \Gamma_{df}$ in the infinitesimal wrapped Fukaya category.}.  Now $k_{\{t\geq f\}}$ corresponds to $\Gamma_{df}$ and we can define a presheaf $F_L^{pre}$ by its sections $\Gamma(U\times (-\infty, c); F_L^{pre}) \coloneqq \colim_f \Hom(\Gamma_{df}, L)$.  The sheafification of $F_L^{pre}$ will give $F_L$ up to the constant sheaf with stalk $H^*(N,k)$.  As we said it is already proved in~\cite{Viterbo-Sheaves} that the associated sheaf $F_L$ has the expected properties; in particular its (reduced) microsupport is $L$.  The problem is to check that $L \mapsto F_L$ is a functor defined on the Fukaya category $\Fuk^\infty$.

Now the language of $\infty$-categories allows us to reformulate the construction directly as a functor. The quantization functor 
\[Q \colon \Fuk^\infty \to \Sh(N \times\Real)\] will be given by the following composition
\begin{align}
  \label{eq:defQ-1}
  \Fuk^\infty  & \to \Fun((\Fuk^\infty)^\op, \fCh(k)) \\
  \label{eq:defQ-2}
  & \to  \Fun((C_\leq^{\infty})^\op, \Derf(k)) \\
  \label{eq:defQ-3}
  & \to  \Fun(\Op^\infty(N)^\op, \Derf(k)) \\
  \label{eq:defQ-4}
  & \to  \Sh(N \times\Real)  ,
\end{align}
that we explain now.  The map~\eqref{eq:defQ-1} is the Yoneda embedding.  To define~\eqref{eq:defQ-2} we compose with $\Gr$ ~\eqref{eq:functions-to-Fukaya} and the localization map $\fCh(k) \to \Derf(k)$.  The last map~\eqref{eq:defQ-4} is explained in~\eqref{eq:pshN-shNR}. It remains to explain~\eqref{eq:defQ-3}.  The functor $C_\leq\to \Op(N)$, $f\mapsto \{f>0\}$, induces $j \colon (C_\leq^\infty)^\op \to (\Op^\infty(N))^\op$.  Since $\Derf(k)$ is presentable, we can consider its left Kan extension (\cite[Corollary 7.3.5.2]{kerodon})
 \[
\operatorname{Lan}_{j} \colon \Fun((C_\leq^\infty)^\op,\Derf(k)) \to \Fun((\Op^\infty(N))^\op, \Derf(k)) .
\]
This is our map~\eqref{eq:defQ-3}.

To make the link with~\cite{Viterbo-Sheaves} we describe how $Q$ acts on objects.  We recall that the left Kan extension is given by a colimit as follows: $\operatorname{Lan}(j)(\phi)(U) = \colim_f \phi(f)$, where $f\in(C_\leq^\infty)^\op$ runs over the functions such that $U \subset \{f>0\}$ (equivalently $f|_U>0$). Switch $C_\leq^\infty$ and $(C_\leq^\infty)^\op$, we also have $\operatorname{Lan}(j)(\phi)(U) = \lim_{f} \phi(f)$, where $f \in C_\leq^\infty$ and $f|_U>0$. In the case where $\phi = \mathcal{Y}_L$ is Yoneda module of a Lagrangian $L$, we write $Q^{pre}(L)=\operatorname{Lan}(j)(\mathcal{Y}_L) $ where $\mathcal{Y}_L(f) = {\Fuk}( \Gamma_{df}, L)$ and we recover the definition of~\cite{Viterbo-Sheaves}.

Lastly, we notice that by the microlocal cut-off lemma and the compactness of $L$ (see Section 3), the functor $Q$ factors through 
\[Q \colon \Fuk^\infty \to \Tam(N)\hookrightarrow \Sh(N \times\Real).\]

\begin{rem}In fact, the composition of \[\widetilde{Q}\colon \fMod(\Fuk^\infty)=\Fun((\Fuk^\infty)^\op, \fCh(k))\rightarrow \Sh(N \times\Real)\] could also be understood as a sheaf quantization functor, and then $Q(L)=\widetilde{Q}(\mathcal{Y}_L)$. A natural question is: for $M\in \fMod(\Fuk^\infty)$, what is the reduced microsupport of $\widetilde{Q}(M)$? For example, when $M$ comes from any kind of Floer theory for singular Lagrangians.
\end{rem}

\subsection{Properties of sheaf quantization functor} Here, we recall some properties of the sheaf quantization functor that is proven in~\cite{Viterbo-Sheaves}, which follow verbatim from arguments therein after adapting our constructions here. Notice that below, we map hom complex of the $A_\infty$-category $\Fuk$, say $\Fuk(L_1,L_2) \in \fCh(k)$, to an object in $ \Derf(k)$ via \eqref{eq:filt-cplx-sheaves}. Notice that the general Floer theory tells that the action filtration of $\Fuk(L_1,L_2)$ satisfies the semi-continuity condition. It follows from \cite[Theorem A-(2)]{APT-Kuwagaki-Zhang} that under \eqref{eq:filt-cplx-sheaves}, we may also regard $\Fuk(L_1,L_2) \in \Derf(k)$ as an object of $\Sh_{\tau\geq0} ( \Real)\hookrightarrow \Derf(k)$.

By the identification $\Sh_{\tau\geq0}(N\times \Real)=\Sh(N; \Sh_{\tau\geq0}(\Real))$, we can define $\Gamma \colon \Sh_{\tau\geq0}(N\times \Real)\rightarrow \Sh_{\tau\geq0}(\Real)$ as the $\Sh_{\tau\geq0}(\Real)$-valued global section functor. More precisely, $\Gamma$ is identified to the $\Sh_{\tau\geq0}(\Real)$-linear direct image to a point $(a_{N})_*^{\Sh_{\tau\geq0}(\Real)}$ for $a_N\colon N\to \pt$, and $\Gamma$ is also equivalent to the $k$-linear direct image functor $(a_N\times \id_{\Real})_*^{k}\colon \Sh_{\tau\geq0}(N\times \Real)\rightarrow \Sh_{\tau\geq0}(\Real)$.

\begin{prop}[{~\cite[Prop. 8.1]{Viterbo-Sheaves}}]\label{prop: sheaf=presheaf} The natural morphism induced by counit of sheafification $Q^{pre}(L) \rightarrow Q(L)$ is an equivalence in $\PSh(N\times\Real)$, i.e. $Q^{pre}(L)$ is already a sheaf. In particular, for any $U\subset N$ open, 
\[FC^*(v^*U,L)\coloneqq\colim_f FC^*(\Gamma_{df}, L) \xrightarrow{\simeq} \Gamma(U,Q(L)) \]
is an equivalence in $\Sh_{\tau\geq0} ( \Real)$. 
\end{prop}

The following is a slight generalization of {~\cite[Coro. 8.2]{Viterbo-Sheaves}}

\begin{cor}\label{cor: FC for conormal}Let $Z$ be a closed submanifold in $N$. Then for a closed exact Lagrangian $L$, we have
\[\Fuk(v^*Z,L)\xrightarrow{\simeq} \Gamma(Z ,Q(L))\]
is an equivalence in $\Sh_{\tau\geq0} ( \Real)$ that is functorial with respect to $L$. \end{cor}
\begin{proof}In the transverse case, we apply Proposition~\ref{prop: sheaf=presheaf} to a $\delta$-tubular neighborhood $U_\delta$ of $Z$, and then take the colimit over $\delta>0$. This is exactly {~\cite[Coro. 8.2]{Viterbo-Sheaves}}.

In the non-transverse case, we use a similar argument to Proposition~\ref{prop: fully faithful quantization} below by perturbing $v^*Z$ by $C^2$-small Hamiltonians.
\end{proof}
\begin{rem}\label{rem: stalk gamma limit} We consider the particular case $Z= \{x\}$ for some point $x\in N$. Then Proposition~\ref{cor: FC for conormal} implies that $\Gamma(Z ,Q(L))=i_x^{-1}Q(L) \in \Sh_{\tau\geq 0}(\Real)$ (with $i_x\colon \Real \rightarrow N\times \Real$) is constructible for generic $x$ since $\Fuk(v^*Z,L)=FC^*(T^*_xN,L)$, and is a $\gamma_\tau$-limit of constructible sheaves for all $x$ since 
$\Fuk(T^*_{x_n}N,L)$ converges to $\Fuk(T^*_xN,L)$ when $x_n$ tends to $x$. 
\end{rem}

The sheaf convolution $\cstar$ and its adjoint $\homst$ define closed symmetric monoidal structures on $\Sh_{\tau\geq0} (N\times \Real)\hookrightarrow \Sh  (N\times \Real)$ that is $\Sh_{\tau\geq0} (\Real)$-linear by the identification $\Sh_{\tau\geq0} (N\times \Real)\simeq \Sh(N;\Sh_{\tau\geq0} ( \Real))$. Then {~\cite[Prop. 9.3, 9.8]{Viterbo-Sheaves}} and their proof imply the following result.

\begin{prop}\label{prop: dualizibility} For any closed exact Lagrangian $L$, the sheaf quantization $Q(L)$ is dualizable with respect to $\cstar$ whose dual is identified with $Q(-L)$. In particular, for any $\F\in \Sh_\geq (N\times \Real)$, we have equivalence $Q(-L)\cstar \F \simeq \homst (Q(L),\F)$ that is functorial with respect to $\F$.
\end{prop}

The following is a slight generalization of {~\cite[Prop. 9.9]{Viterbo-Sheaves}}

\begin{prop}\label{prop: fully faithful quantization}The sheaf quantization functor $Q$ induces functorial equivalences
\[\Fuk(L_1,L_2)\xrightarrow{\simeq }\Gamma(N,\homst(Q(L_1),Q(L_2)))  \]
in $\Sh_{\tau\geq0} ( \Real)$ for all pairs $(L_1,L_2)$. \end{prop}

In particular, the sheaf quantization functor induces a fully faithful functor $\Derf(\Fuk^\infty)_0\rightarrow \Sh_{\tau\geq0}(N\times \Real)$ where $\Derf(\Fuk^\infty)_0$ is the full subcategory of Yoneda modules in $\Derf(\Fuk^\infty)$. 
\begin{rem}The homotopy category of $\Derf(\Fuk^\infty)_0$ is exactly the filtered version of Donaldson-Fukaya category since over fields two $A_\infty$ modules are quasi-isomorphic if and only they are homotopy equivalent. 
\end{rem}
\begin{proof}We first explain the equivalences for each pair $(L_1,L_2)$.

The proposition {~\cite[Prop. 9.9]{Viterbo-Sheaves}} prove the statement for transversely intersected pairs $(L_1,L_2)$. To the reader's convenience, let us provide its proof here. 

By dualizability Proposition~\ref{prop: dualizibility}, the right-hand side is computed by $\Gamma(N,Q(-L_1)\cstar Q(L_2)) \simeq \Gamma(\Delta_N,Q(-L_1)\boxstar Q(L_2))$. The uniqueness of sheaf quantization shows that $Q(-L_1)\boxstar Q(L_2)\simeq Q(-L_1 \times L_2)$, where the later is the quantization in $T^*(N\times N)$. Then we apply Corollary~\ref{cor: FC for conormal} to $Z=\Delta_N$, and identify $FC(\Delta_N,-L_1 \times L_2)=FC( L_1 , L_2)$.

For general pairs $(L_1,L_2)$, we pick a sequence $C^2$-small Hamiltonians $H_n$ for $\varphi_n=\varphi_{H_n}$ such that $\varphi_{n}(L_1)$ are transversely intersected with $L_2$.

On the sheaf side, we take $\mathcal{K}_{n}$ the sheaf quantization of $H_n$ (at time $1$). By \cite[Appendix B]{Ike_2019} we have $Q(\varphi_{n}(L_1))\simeq \mathcal{K}_{n}  (Q(L_1))$, and then \[\Gamma(N,\homst(Q(\varphi_{n}(L_1)),Q(L_2))) \simeq \Gamma(N,\homst(\mathcal{K}_{n}  (Q(L_1)),Q(L_2))) \] is convergent to $\Gamma(N,\homst(Q(L_1),Q(L_2))) $ with respect to the interleaving distance by \cite[Prop. 4.11, Thm. 4.16]{Asano-Ike}. Similarly, on the Fukaya side, we have $\Fuk(\varphi_{n}(L_1),L_2)$ is convergent to $\Fuk(L_1,L_2)$ by the same Hofer distance estimation via applying \cite[Theo. 3.-4-(i)]{Biran-Cornea-Zhang} to $\Fuk$. 

Then the equivalence for the general pairs $(L_1,L_2)$ follows from the transversely intersected pairs and the limit argument.

To see that $Q$ induces a fully faithful functor $\Derf(\Fuk^\infty)_0\rightarrow \Sh_{\tau\geq0}(N\times \Real)$, we shall check that $Q$ induces homotopy equivalence between the ($\Sh_{\tau\geq0} ( \Real)$-enriched) hom objects of $\Derf(\Fuk^\infty)_0$ and of $\Sh_{\tau\geq0}(N\times \Real)$. Now, we explain the ingredients and the proof for the claim.

By the enriched Yoneda lemma, the hom objects of $\Derf(\Fuk^\infty)_0$ is $\Fuk(L_1,L_2)  \in \Sh_\geq ( \Real)\hookrightarrow \Derf(k)$. And the hom object of $\Sh_{\tau\geq0}(N\times \Real)$ is $\Gamma(N,\homst(\F_1,\F_2))$. At this moment, we already know $Q$ is a functor, which induces the natural map $\Fuk(L_1,L_2)\rightarrow\Gamma(N,\homst(Q(L_1),Q(L_2))) $. It remains to check it is an equivalence on cohomology.

We describe the natural map on cohomology: For a morphism $x\colon  L_1\rightarrow L_2$, we have a morphism $\mathcal{Y}_{L_1}\rightarrow \mathcal{Y}_{L_2}$ of Yoneda modules induced by the pair of pants product, and induce a morphism  $Q(L_1)\rightarrow Q(L_2)$ since $Q$ are sheafification of Yoneda modules (restricted to exact graphs). 

However, it is shown that in \cite[Section 10]{Viterbo-Sheaves} the equivalences we described in the first part are exactly induced by the pair of pants product on cohomology.
\end{proof}

\section{Strategy of proof}\label{sec:strategyproof}

We now explain how we will use the quantization functor $Q$ to prove the main theorem.  We first sum up the results recalled so far and add some notations. Then we explain that it will be useful to consider the projector from $\Tam(N)$ to $\Tam_{DT^*N}(N)$ to obtain the result.

\addtocontents{toc}{\SkipTocEntry}
\subsection*{Some reminder}
We recall that we consider $\Fuk^\infty(DT^*N)$ as an $\infty$-category.  We consider the filtered derived category \[\Derf(\Fuk^\infty(DT^*N)) = \Fun(\Fuk^\infty(DT^*N)^{op}, \Derf(k))\] and we let $\mathcal Y^F\colon \Fuk^\infty(DT^*N) \to \Derf(\Fuk^\infty(DT^*N))$, be the $\Derf(k)$-enriched Yoneda embedding. For $x\in N$, $a\in \Real$ we denote by $V_{(x,a)}$ the fiber at $x$ with primitive $a$ of $\lambda_{\mid V_x}=0$ and we define $\V_{(x,a)} \in \fMod(\Fuk^\infty(DT^*N))$ as the functor $L' \mapsto FC^*(L', V_{(x,a)})$, and we use the same notation $\V_{(x,a)}$ to represent the corresponding object in $\Derf(\Fuk^\infty(DT^*N))$. 

We have recalled in~\S\ref{sec:quantization} the quantization functor $Q\colon \Fuk^\infty(DT^*N) \to \Tam(N)$.  The functor $Q$ commutes with the translation functors, $T_c$ on $\Fuk^\infty(DT^*N)$ and $T_{c*}$ on $\Tam(N)$. By construction $\rho(SS(Q(L))$ is $L$, and $Q$ actually takes values in $\Tam_{DT^*N}(N)$.  The composition with $Q$ induces an exact functor $Q^* \colon \Derf(\Tam_{DT^*N}(N)) \to \Derf(\Fuk^\infty(DT^*N))$, $E \mapsto E \circ Q$.

We let $i\colon \Tam_{DT^*N}(N) \to \Tam(N)$ be the embedding and $i^*$
be the functor induced on the filtered derived categories.
We let $\mathcal{Y}^s \colon \Tam(N) \to \Derf(\Tam(N))$, $\mathcal{Y}^s_1 \colon \Tam_{DT^*N}(N) \to \Derf(\Tam_{DT^*N}(N))$ be the $\Derf(k)$-enriched Yoneda functors. Note that so far we have four Yoneda functors: $\mathcal Y$ valued in $\Mod^{fil}(\Fuk^\infty (DT^*N))$, $\mathcal Y^F$ valued in $\Derf(\Fuk^\infty (DT^*N))$ and $\mathcal Y^s, \mathcal Y_1^s$ on the Tamarkin side, which we regard $\mathcal{Y}^F, \mathcal{Y}^s_1, \mathcal{Y}^s$ as $\Derf(k)$-enriched Yoneda functors. We obtain the commutative diagram
\[
\xymatrix@C=1.5cm{
  \Fuk^\infty(DT^*N)  \ar[r]^Q  \ar[d]_{\mathcal{Y}^F}  & \Tam_{DT^*N}(N) \ar[r]^i \ar[d]_{\mathcal{Y}^s_1} & \Tam(N)  \ar[d]_{\mathcal{Y}^s}  \\
  \Derf(\Fuk^\infty(DT^*N)) & \Derf(\Tam_{DT^*N}(N)) \ar[l]^{Q^*} & \Derf(\Tam(N)) \ar[l]^-{i^*}
  }
\]

\begin{prop} We have equivalences in $\Derf(\Fuk^\infty(DT^*N))$:
\begin{equation}\label{eq:fibers}
  Q^* i^* \mathcal{Y}^{s}(k_{\{x\}\times [a,+\infty)}) \simeq \V_{(x,a)} .
\end{equation}
\end{prop}
\begin{proof}By identifying $\Fuk^\infty(DT^*N)$ with its opposite $\Fuk^\infty(DT^*N)^\op$ via $L\mapsto -L$, we can identify the Yoneda modules $\V_{(x,a)}\colon L' \mapsto FC^*(L', V_{(x,a)})$ with the (anti-)coYoneda modules, $L' \mapsto FC^*(V_{(x,a)},-L')$. On the other hand,
$FC^*(V_{(x,a)},-L')\simeq  \Gamma(T_{x}N ,T_{a}Q(-L')) \simeq (Q(-L'))_{(x,a)}$, by Corollary~\ref{cor: FC for conormal}, where the right-hand side is represented by $k_{\{x\}\times [a,+\infty)}$ in $\Tam(N)$. We notice that all equivalences are functorial in $L'$ with respect to evident natural module morphisms. 
\end{proof}

The categories $\Fuk^\infty(DT^*N)$, $\Tam_{DT^*N}(N)$ and $\Tam(N)$ come with translation functors, hence interleaving distances.  By Remark~\ref{rem:preserve distance} their derived categories inherit translation functors and distances.  We denote all these distances by $\gamma$ but when we want to be more precise we let $\gamma^F$, $\gamma^s_1$, $\gamma^s$ be the distances on $\Fuk^\infty(DT^*N)$, $\Tam_{DT^*N}(N)$, $\Tam(N)$ and we let $\hat\gamma^F$, $\hat\gamma^s_1$, $\hat\gamma^s$ be the distances on the derived categories.

The functors $\mathcal{Y}^F, \mathcal{Y}^s_1, \mathcal{Y}^s$ are fully faithful and commute with the translation functors, hence they preserve the distances. However $Q^*$ and $i^*$ are only $1$-Lipschitz (although $Q, i$ are isometric embeddings):
\begin{equation}
  \label{eq:Q_i_Lipschitz}
  \hat\gamma^F(Q^*(F), Q^*(G)) \leq \hat\gamma^s_1(F, G),
  \qquad
  \hat\gamma^s_1(i^*(F), i^*(G)) \leq \hat\gamma^s(F, G) .
\end{equation}

\addtocontents{toc}{\SkipTocEntry}
\subsection*{Projector}

We recall that we want to prove the following density statement: for a given $L \in \Fuk^\infty(DT^*N)$ and a given $\varepsilon > 0$ there exists finitely many fibers $\V_{(x_1, a_1)},\dots,\V_{(x_p, a_p)}$ and an iterated cone $C$ of their Yoneda modules $\V_{(x_1, a_1)},\dots,\V_{(x_p, a_p)}$ such that $\hat\gamma^F(\mathcal{Y}^F(L), C) < \varepsilon$.  In view of the isomorphism~\eqref{eq:fibers} and the Lipschitz bounds~\eqref{eq:Q_i_Lipschitz} it would be enough to have such a density result in $\Tam(N)$, that is, ``there exists an iterated cone $\F$ of the $k_{\{x_i\}\times [a_i,+\infty)}$'s such that $\gamma^s(i Q (L), \F) < \varepsilon$''.  However this latter statement cannot be true because there is no non-zero morphism in $\Tam(N)$ between $k_{\{x_i\}\times [a_i,+\infty)}$ and $k_{\{x_j\}\times [a_j,+\infty)}$ as soon as $x_i \not= x_j$.

The crucial point is that the composition $i^* \mathcal{Y}^s$ factors through $\Tam_{DT^*N}(N)$.  Indeed the functor $i$ has both a left and a right adjoint, by~\cite{Kuo-wrappedsheaves} and~\cite{Kuo-Shende-Zhang-Hochschild}. More precisely \cite{Kuo-wrappedsheaves}  proves the existence of adjoints of the embedding $\Sh_Z(N) \to \Sh(N)$ (denoted $W^+$ and $W^-$ in loc. cit. and constructed by a wrapping procedure) in the case where $N$ is compact and $Z$ is a closed conic subset of $T^*N$ ($\Sh_Z(N)$ denotes the category of sheaves with microsupport contained in $Z$). These results are extended to our setting (the non compact manifold $N\times \Real$) in~\cite{Kuo-Shende-Zhang-Hochschild}. Let us denote by $P'_{DT^*N}$ the right adjoint to $i$. 

We then have, for any $\F \in \Tam_{DT^*N}(N)$ and $\G \in \Tam(N)$
\begin{equation}\label{adjunction-formula}
i^* \mathcal{Y}^s(\G)(\F) = \homst_{\Tam(N)}( i(\F), \G) \simeq \homst_{\Tam_{DT^*N}(N)}(\F, P'_{DT^*N}(\G))
\end{equation}
which shows that $i^* \mathcal{Y}^s(\G)$ is represented in $\Tam_{DT^*N}(N)$ by $P'_{DT^*N}(\G)$.  We thus also have the commutative diagram
\[
\xymatrix@C=1.5cm{
  \Fuk^\infty(DT^*N)  \ar[r]^Q  \ar[d]_{\mathcal{Y}^F}  & \Tam_{DT^*N}(N) \ar[d]_{\mathcal{Y}^s_1} & \Tam(N)  \ar[d]_{\mathcal{Y}^s} \ar[l]_-{P'_{DT^*N}}  \\
  \Derf(\Fuk^\infty(DT^*N)) & \Derf(\Tam_{DT^*N}(N)) \ar[l]^{Q^*} & \Derf(\Tam(N)) \ar[l]^-{i^*}
  }
\]
Let us set $\W_{(x,a)} = P'_{DT^*N}(k_{\{x\} \times [a, \infty)})$.  Then~\eqref{eq:fibers} translates into
\begin{equation}\label{eq:QprimeW is V}
  Q^*(\W_{(x,a)}) \simeq \V_{(x,a)}
\end{equation}
and Theorem~\ref{thm:density Fukaya} follows from the following density statement for sheaves. 

\begin{thm}[Density Theorem for sheaves]\label{thm:density sheaves}
Let $\F \in \Tam_{DT^*N}(N\times\Real)$ such that for any $x\in N$ the sheaf $\F \otimes k_{\{x\} \times\Real}$ is a $\gamma$-limit of constructible sheaves on $\Real$. Then for any $\varepsilon >0$, there are points $(x_i)_{i\in \{1, \ldots,l\}}$, real numbers $(a_i)_{i\in \{1, \ldots,l\}}$ and $C \in \langle \W_{(x_1,a_1)}, \ldots, \W_{(x_l,a_l)} \rangle$ in ${\Tam_{DT^*N}(N\times\Real)}$ such that $\gamma^s (\F,C) < \varepsilon$.
\end{thm}

To apply it to Theorem~\ref{thm:density Fukaya}, we note that by Remark~\ref{rem: stalk gamma limit}, $Q(L) \otimes k_{\{x\} \times\Real}$ is a $\gamma$-limit of constructible sheaves on $\Real$ for each $x$.

\begin{rem}
The hypothesis on the sheaf $\F$ in Theorem~\ref{thm:density sheaves} is weaker than ``$\F$ is a $\gamma$-limit of constructible sheaves''. Indeed, since we assume that its microsupport is contained in $\rho^{-1}(DT^*N)$ (which can be seen as a quantitative version of ``non-characteristic with respect to the fibers $\{x\} \times\Real$''), we can bound the distance on $\{x\} \times\Real$ by the distance on $N \times \Real$ using Proposition~\ref{prop:compare_dist} below.
\end{rem}

\section{Distances on the category of sheaves}

In this section we consider a manifold $M$, which will be either $N$ or $N\times\Real$, with $N$ a closed manifold.
We recall and extend some results about interleaving distances on $\Sh(M)$.  The interleaving distance for sheaves was first introduced in~\cite{K-S-distance} and some variations are considered in~\cite{Asano-Ike} and~\cite{Guillermou-Viterbo}.  Any non-negative homogeneous Hamiltonian function on $T^*M \setminus 0_M$ induces a translation functor on $\Sh(M)$.  This was first noticed by Tamarkin~\cite{Tamarkin} in the special case $M = N\times \Real$ with $T_c$ the direct image by the translation in the $\Real$ direction.

\subsection{Sheaf distance for a continuous norm}
Let $h\colon T^*M\setminus 0_M \to \Real$ be a smooth function which is positively homogeneous in the sense that $h(x, \lambda p) = \lambda h(x,p)$ for $\lambda>0$.  We set for short $\Sh_{h\geq 0}(M) = \Sh_{\{(x,p)\mid h(x,p)\geq0\}}(M)$.  

To such an $h$ we associate a distance $\gamma_h$ on $\Sh_{h\geq 0}(M)$ defined in the following way.  We let $\varphi_h$ be the Hamiltonian flow of $h$ and $\K_{\varphi_h} \in \Sh(M^2 \times \Real)$ the sheaf defined by $\varphi_h$ in~\cite{G-K-S}. Very often we write for short $\K_h = \K_{\varphi_h}$. We recall that $\K_{\varphi_h}$ is defined so that $SS(\K_{\varphi_h}^a)$ is the graph of $\varphi_h^{-a}$, where $\K_{\varphi_h}^a$ is the restriction of $\K_{\varphi_h}$ to $M^2\times \{a\}$ -- more precisely (use~\cite[Lem. A.1]{G-K-S}, applied to an autonomous Hamiltonian), away from the zero section,
\begin{equation}
\label{eq:microsupport_K_vraphi_h}
  SS(\K_{\varphi_h}) = \big\{ (\varphi_h(x, p), x, -p, s, -h(x,p))  \mid  (x,p) \in T^*M, s \in\Real \big\} .
\end{equation}
Here we point out that the order of composition of sheaves can be misleading: if $f\colon X \to Y$ is a function, with graph $\Gamma_f \subset X\times Y$, and $y\in Y$, then $k_{\Gamma_f} \circ k_{\{y\}} = k_{\{f^{-1}(y)\}}$.  So, if we want that $SS(\K_h^s \circ \F) = \phi_h^s(SS(\F))$, we have to be careful that $SS(\K_h^s)$ should be the graph of $(\phi_h^s)^{-1}$.

For a given $a$, the composition with $\K_{\varphi_h}^a$, also sometimes denoted $\K_{\varphi_h}^a \colon \Sh(M) \to \Sh(M)$, $\F \mapsto \K_{\varphi_h}^a(\F) = \K_{\varphi_h}^a \circ \F$, preserves the subcategory $\Sh_{h\geq 0}(M)$.  Moreover, if we consider the whole family parametrized by $a$, the functor $\K_{\varphi_h} \colon \Sh(M)\rightarrow \Sh(M\times \Real)$, $\F \mapsto \K_{\varphi_h} \circ \F$, restricts to a functor $\Sh_{h\geq 0}(M)\rightarrow \Sh_{\alpha \leq 0}(M\times \Real)$, where $\alpha$ is the dual \footnote{The symplectic form on $T^*(M\times \Real)$ is $dp\wedge dq+d\alpha\wedge da$} of $a\in \Real$. This follows from~\eqref{eq:microsupport_K_vraphi_h} since the variable $\alpha$ is given by $-h$ (see also~\cite[Prop. 4.8]{G-K-S}).

The restrictions of $\K_{\varphi_h}^a(-)$ to $\Sh_{h\geq 0}(M)$ (for all $a$) come with the natural Tamarkin morphisms $\tau_{a,b}^h(\F) \colon \K_{\varphi_h}^a(\F) \to \K_{\varphi_h}^b(\F)$ for $a\leq b$. Let us recall briefly their construction. On the category $\Sh_{\alpha \leq 0}(M\times \Real)$, the restriction to $M\times \{a\}$ can be computed by $\F|_{a}\simeq \pi_{M!}(\F\otimes k_{(-\infty, a)}[1])$ as can be seen by applying the functor $\pi_{M!}(\F\otimes -)$ to the triangle $k_{(-\infty, a]} \to k_{\{a\}} \to k_{(-\infty, a)}[1] \to[+1]$ and using the fact that $\Gamma_c(\Real; \G) \simeq 0$ for a sheaf $\G \in \Sh_{\alpha \leq 0}(\Real)$ with support contained in $(-\infty, a]$. Then the open inclusion $(-\infty,a)\subset (-\infty,b)$ induces a natural morphism $\tau_{a,b}(\F)\colon  \F|_{a}\rightarrow \F|_{b}$ for $a\leq b$ and $\tau_{a,b}^h(\F)$ is constructed as
\[\xymatrix@C=3cm{
  \K_{\varphi_h}(\F)|_{a} \ar[d]^{\simeq} \ar[r]^-{\tau_{a,b}(\K_{\varphi_h}(\F))} & \K_{\varphi_h}(\F)|_{b} \ar[d]^{\simeq}   \\
   \K_{\varphi_h}^a(\F)  \ar[r]^-{\tau_{a,b}^h(\F)} & \K_{\varphi_h}^a(\F)
}\]

As noted in~\cite{Guillermou-Viterbo}, the data $(\K_{\varphi_h}^a, \tau_{0,a}^h )_{a\in \Real}$ defines an $\arbR$-action on $\Sh_{h\geq 0}(M)$ (since $\tau_{a,b}$ and $\K_{\varphi_h}^a$ do),   and it defines an interleaving type pseudo-distance on $\Sh_{h\geq 0}(M)$ that we denote by $\gamma_h$: for $\F,\G \in \Sh_{h\geq 0}(M)$,
\begin{equation} 
  \label{eq:def_distance}
  \begin{aligned}
   \gamma_h(\F,\G) &= \inf \Big\{a+b \mid a,b \geq 0,\; \exists u\colon\F \to \K_{\varphi_h}^b(\G), v\colon \G \to \K_{\varphi_h}^a(\F), \\
                    &  \K_{\varphi_h}^{a}(u) \circ v \colon \G \to \K_{\varphi_h}^{a+b}\G, \; \K_{\varphi_h}^{b}(v)\circ u\colon \F \to \K_{\varphi_h}^{a+b}\F  \text{are homotopic to $\tau^h_{0,a+b}(-)$} \Big\}
  \end{aligned}
\end{equation}

\medskip

It will be useful in this paper to extend the definition of $\gamma_h$ to the case where $h$ is only continuous.  For example, for two homogeneous functions $h\colon T^*M\setminus 0_M \to \Real$, $h'\colon T^*M'\setminus 0_N \to \Real$ we define the sum $h+h'$ by $(h+h')(x,p, x',p') = h(x,p) + h'(x',p')$ on $(T^*M\setminus 0_M) \times (T^*M'\setminus 0_{M'})$ and extend it by continuity to $T^*(M\times M') \setminus 0_{M\times M'}$.  In general $h+h'$ is not a $C^1$ function.

To define the distance $\gamma_h$ it is enough to know the sheaf $\K_{\varphi_h}$. We claim that a similar sheaf can be naturally defined for a continuous $h$, by approximating $h$ with smooth functions.  Let us introduce some notations.  We recall that we assume that $M$ is closed, or $M = N\times \Real$ with $N$ closed.  We let $\C^\infty_{ph}(T^*M)$ (resp. $\C^0_{ph}(T^*M)$) be the space of smooth (resp. continuous) positively homogeneous ($ph$ stands for positively homogeneous) functions on $T^*M$; in the case $M = N \times \Real$ we ask that the functions $h$ in $\C^\infty_{ph}(T^*M)$, or $\C^0_{ph}(T^*M)$, do not depend on the variable $t$ on $\Real$ (so for $\tau\neq 0$, $h(x,p,t,\tau) = \tau h_0(x, p/\tau)$ for some function on $T^*N$).

Let us pick a metric on $N$. In case $M= N \times \Real$ we endow $M$ with the product metric.
\begin{lem}\label{lem:def_Kh}
  We consider the non-negative function $g$ on $T^*(M^2 \times \Real) \setminus 0_{M^2\times \Real}$ defined by $g(x$, $p$, $x'$, $p'$, $a, \alpha) = (\sqrt{1+a^2} \|(p, p')\|^2 + \alpha^2)^{1/2}$ (we replace $\|(p,p')\|$ by $\|(p, \tau, p', \tau')\|$ in the case $M=N\times\Real$). Let $h \in \C^0_{ph}(T^*M)$ be given and let $(h_n)_n$ be a sequence in $\C^\infty_{ph}(T^*M)$ $C^0$-converging to $h$. Then the sequence $\K_{\varphi_{h_n}}$ is Cauchy with respect to the distance $\gamma_g$. The limit
  \[
  \K_h \coloneqq \gamma_g-\lim_n \K_{\varphi_{h_n}}
\]
is well-defined and independent of the sequence $(h_n)_n$.
\end{lem}
\begin{proof}
  The first assertion follows for example from~\cite[Thm. 1.13]{Asano-Ike-Kuo-Li25} which says that,  for the function $g' = \|(p,p')\|$, $\gamma_{g'}(k_{\Delta_M}, \K_{h_n}^1)$ is bounded by the Hofer norm of $h_n$. We note that in loc. cit. the authors consider the compact case and they consider time $1$ and not all times together as we do here. Now their proof extends to our non compact situation which is $M = N\times \Real$, taking into account the fact that the Hamitlonians are independent of the variable $t$: the proof is actually given in Lemma~3.1 of~\cite{Asano-Ike-Kuo-Li25} and the reader can see that the compactness hypothesis is only used to have finite bounds on $h_n/g$. The result for time $1$ also extends because we have rescalled the metric by the factor $\sqrt{1+a^2} \sim |a|$ ($|a|\to \infty$). Indeed $\K_{h_n}^a \simeq \K_{ah_n}^1$, hence $\gamma_{|a|g'}(k_{\Delta_M}, \K_{h_n}^a) = (1/|a|)\gamma_{g'}(k_{\Delta_M}, \K_{ah_n}^1)$ is bounded by the Hofer norm of $h_n$.
  
  The fact that Cauchy sequences have a limit follows from~\cite[Prop. 6.22]{Guillermou-Viterbo} or~\cite{Asano-Ike24_Completeness}.  It is clear that the distance between two different limits (possibly for different Cauchy sequences $(h_n)_n$) is $0$. The distance $\gamma_g$ is degenerate but, if we endow $M$ with a real analytic structure and choose a real analytic metric, then \cite[Prop. 6.22]{Guillermou-Viterbo} says that $\gamma_g$ is non-degenerate when we restrict to sheaves that are limit of constructible sheaves.  Now we can assume that the functions $h_n$ are real analytic, which proves that one limit $\K_h$ (hence all) is a limit of constructible sheaves, hence uniquely defined.
\end{proof}

Lemma~\ref{lem:def_Kh} says that $\K_h$ is uniquely defined up to equivalence.  We can also give a functorial construction of $\K_h$ on $M^2 \times [0,\infty)$ as follows, which will also prove that $\K_h$ is independent of the choice of a metric.  We consider $\C^\infty_{ph}(T^*M)$ and $\C^0_{ph}(T^*M)$ as posets for the usual order $\leq$.  If $h \leq h'$ are smooth positively homogeneous functions, then we have a natural morphism, for any $a\geq0$, $\K_{\varphi_{h}}^a \to \K_{\varphi_h'}^a$.  We even have a morphism $\K_{\varphi_{h}}|_{M^2 \times [0,\infty)} \to \K_{\varphi_h'}|_{M^2 \times [0,\infty)}$.  By~\cite{Kuo-wrappedsheaves} the morphisms $\K_{\varphi_{h}}|_{M^2 \times [0,\infty)} \to \K_{\varphi_h'}|_{M^2 \times [0,\infty)}$ organize into a functor from the poset $\C^\infty_{ph}(T^*M)$ to the category $\Sh(M^2\times [0,\infty))$.  For a given $h_0 \in \C^0_{ph}(T^*M)$ we let $\C^\infty_{ph, < h_0}(T^*M)$ be the subposet of $\C^\infty_{ph}(T^*M)$ of functions $h$ such that $h(z) < h_0(z)$ for any $z$.  Any increasing sequence $(h_n)_n$ in $\C^\infty_{ph, < h_0}(T^*M)$ which converges to $h_0$ is cofinal in $\C^\infty_{ph, < h_0}(T^*M)$. By~\cite[Lem. 6.21]{Guillermou-Viterbo} we also have $\K_{h_0}|_{M^2 \times [0,\infty)} \simeq \colim_n \K_{\varphi_{h_n}}|_{M^2 \times [0,\infty)}$. Hence we obtain a functorial description of $\K_{h_0}|_{M^2 \times [0,\infty)}$
\begin{equation}
  \label{eq:def_Kh_continuous_colimit}
  \K_{h_0}|_{M^2 \times [0,\infty)} = \colim_h \K_{\varphi_{h}}|_{M^2 \times [0,\infty)}, 
\end{equation}
where $h$ runs over $\C^\infty_{ph, < h_0}(T^*M)$.  From~\eqref{eq:def_Kh_continuous_colimit} we see that $h \mapsto \K_{h}|_{M^2 \times [0,\infty)}$ extends the functor $h \mapsto \K_{\varphi_{h}}|_{M^2 \times [0,\infty)}$, initially defined on $\C^\infty_{ph}(T^*M)$, to a functor from $\C^0_{ph}(T^*M)$ to $\Sh(M^2\times [0,\infty))$. We remark that~\eqref{eq:def_Kh_continuous_colimit} still holds if we let $h$ runs over the poset $\C^0_{ph, < h_0}(T^*M)$ which we define like $\C^\infty_{ph, < h_0}(T^*M)$.
For negative times we have a similar formula:
$\K_{h_0}|_{M^2 \times (-\infty,0]} = \colim_h \K_{\varphi_{h}}|_{M^2 \times (-\infty,0]}$,
where now $h$ runs over the poset $\C^0_{ph, > h_0}(T^*M)$ of functions greater than $h_0$.

\smallskip

Since colimits of sheaves commute with inverse image, the formula~\eqref{eq:def_Kh_continuous_colimit} implies, for any $a \in \Real$, 
\[
\K_{h_0}^a \coloneqq \K_{h_0}|_{M^2 \times \{a\}} \simeq \colim_h \K_{\varphi_{h}}^a
\]
where $h$ runs over $\C^\infty_{ph, < h_0}(T^*M)$ (resp. $\C^\infty_{ph, > h_0}(T^*M)$) when $a\geq 0$ (resp. $a\leq 0$). For $a=0$ we find $\K_{h_0}^0 \simeq k_{\Delta_M}$ as expected.

\begin{lem}\label{lem:bound SS Kh}
  Let $h \in \C^0_{ph}(T^*M)$.  We define $\Gamma_h \subset T^*(M \times \Real)$ by
  \[
  \Gamma_h = \{ (x, -p, a, -h(x,p)) \mid (x,p) \in T^*M, a\in \Real\}.
  \]
Then $SS(\K_h) \subset T^*M \times \Gamma_h$.  In particular the functor $\K_h \colon \Sh(M)\rightarrow \Sh(M\times \Real)$, $\F \mapsto \K_h \circ \F$, restricts to a functor $\Sh_{h \geq 0}(M)\rightarrow \Sh_{\alpha \leq 0}(M\times \Real)$.
\end{lem}
\begin{proof}
  We pick a sequence $(h_n)_{n\in \N}$ in $\C^\infty_{ph, > h}(T^*M)$ converging to $h$.  We have $SS(\K_h) \subset \liminf SS(\K_{h_n})$ by~\cite[Ex. 5.7]{K-S} (or~\cite[Prop. 6.26]{Guillermou-Viterbo} for details), where $\liminf_j X_j = \left\{x \mid \exists (x_j)_{j\geq 1}, \; x_j\in X_j,\; \lim_j x_j=x\right \}$.  Since $SS(\K_{h_n})$ is given by~\eqref{eq:microsupport_K_vraphi_h}, we have the rough bound $SS(\K_{h_n}) \subset T^*M \times \Gamma_{h_n}$ and we deduce the first assertion.

  Now we have a bound for $SS(\K_h \circ \F)$ given by a ``set theoretic composition'' $SS(\K_h) \circ SS(\F)$ (see~\cite[(1.12)]{G-K-S}). The bound is $\{(x,p, a, \alpha) \mid \exists x', p'$, $(x,p,x',-p', a, \alpha) \in SS(\K_h)$, $(x',p') \in SS(\F)\}$.  The result follows.  
\end{proof}

By Lemma~\ref{lem:bound SS Kh} and the discussion in the beginning of this section, there exist natural morphisms $\tau_a^h(\F) \colon \F \to \K_h^a(\F)$ for $\F\in \Sh_{h\geq 0}(M)$ and $a\geq 0$, which induce an $\arbR$-action on $\Sh_{h\geq 0}(M)$.  We can now define a pseudo-distance as in the smooth case:
\begin{defn}
  Let $h\colon T^*M \to \Real$ be a positively homogeneous continuous function and let $\K_h$ be the sheaf defined in Lemma~\ref{lem:def_Kh}.  Using this sheaf $\K_h$ we define a pseudo-distance $\gamma_h$ on $\Sh_{h\geq 0}(M)$ as in~\eqref{eq:def_distance}.
\end{defn}

\subsection{Some properties of the distance}

\begin{lem}\label{lem:kernel product}
  Let $h\colon T^*M\setminus 0_M \to \Real$, $h'\colon T^*M'\setminus 0_{M'} \to \Real$ be two continuous homogeneous functions.  Then $\K_{h+h'} \simeq \K_{h} \boxtimes_\Real \K_{h'}$, where $\boxtimes_\Real$ denotes the exterior tensor product relative to the factor $\Real$.
\end{lem}
\begin{proof}
  We first assume that $h,h'$ are smooth on $T^*M\setminus 0_M$, $T^*M'\setminus 0_{M'}$.  We pick a sequence $k_n$ in $\C^\infty_{ph}(T^*(M\times M'))$ converging to $h+h'$. We can assume that $k_n = h + h'$ on $U(\varepsilon_n) = (T^*M\setminus D_{\varepsilon_n}T^*M) \times (T^*M'\setminus D_{\varepsilon_n}T^*M')$ for some sequence $(\varepsilon_n)_n$ converging to $0$ and that $dk_n$ is small on $Z_n = T^*(M\times M') \setminus U_n$.  Let us denote by $\Gamma_{\varphi_h}$ the graph of $\varphi_h$ as described in the right hand side of~\eqref{eq:microsupport_K_vraphi_h}.  Then $\Gamma_{\varphi_{k_n}} \cap U(\varepsilon_n) = (\Gamma_{\varphi_h} \times_\Real \Gamma_{\varphi_{h'}}) \cap U(\varepsilon_n)$ and, choosing $dk_n$ small enough on $Z_n$, we find $\liminf_n \Gamma_{\varphi_{k_n}} = \Gamma_{\varphi_h} \times_\Real \Gamma_{\varphi_{h'}}$. Using the ``limit property'' of microsupports as in the proof of Lemma~\ref{lem:bound SS Kh}, we deduce that $SS(\K_{h+h'}) \subset \Gamma_{\varphi_h} \times_\Real \Gamma_{\varphi_{h'}}$.

  To conclude it is enough to prove that there exists a unique $\K \in \Sh((M\times M')^2 \times \Real)$ such that $\K|_{(M\times M')^2 \times \{0\}} \simeq k_{\Delta_{M\times M'}}$ and $SS(\K) \subset \Gamma_{\varphi_h} \times_\Real \Gamma_{\varphi_{h'}}$.  We follow the proof of the uniqueness of the quantization~\cite[Prop. 3.2]{G-K-S}. The main point is that $\K_{\varphi_h}$ is invertible with respect to the composition of sheaves, with inverse $\K_{\varphi_h}^{\otimes -1} \simeq \K_{\varphi_{-h}}$.  We consider the composition relative to the factor $\Real$ and set $\K' = \K \circ_\Real (\K_{\varphi_{h}}^{\otimes -1} \boxtimes_\Real \K_{\varphi_{h'}}^{\otimes -1})$.  Then the bounds for the behaviour of the microsupport under sheaf operations in~\cite[\S5.4]{K-S} (also used in the proof of Lemma~\ref{lem:bound SS Kh}) give $SS(\K') \subset T_{\Delta_{M\times M'}\times \Real}^*((M\times M')^2 \times \Real)$. Since we also have $\K'|_{(M\times M')^2 \times \{0\}} \simeq k_{\Delta_{M\times M'}}$, this implies that $\K' \simeq k_{\Delta_{M\times M'}\times \Real}$ and then $\K \simeq \K_{\varphi_{h}} \boxtimes_\Real \K_{\varphi_{h'}}$.

\smallskip  
  This gives the result when $h,h'$ are smooth.  In general we pick increasing sequences $(h_n)_n$, $(h'_n)_n$ in $\C^\infty_{ph, < h}(T^*M)$ and $\C^\infty_{ph, < h'}(T^*M')$ converging to $h$, $h'$. Then the sequence $(h_n +h'_n)_n$ is cofinal in $\C^0_{ph, < h+h'}(T^*(M\times M'))$.  Hence~\eqref{eq:def_Kh_continuous_colimit} (which also holds when the colimit is taken over continuous functions, as we already noticed) gives the result over $(M\times M')^2 \times [0,\infty)$ because colimits of sheaves commute with $\boxtimes_\Real$.  We have in the same way the result for negative times and this concludes the proof.  
\end{proof}

\begin{lem}\label{lem:bound product}
  Let $h\colon T^*M\setminus 0_M \to \Real$, $h'\colon T^*M'\setminus 0_{M'} \to \Real$ be two continuous positively homogeneous functions.  Let $\F,\G \in \Sh_{h\geq 0}(M)$, $\F', \G' \in \Sh_{h'\geq 0}(M')$. Then $\F\boxtimes \F' \in \Sh_{h+h'\geq 0}(M\times M')$ and
  \begin{equation}
    \label{eq:bound product}
    \gamma_{h+h'}(\F\boxtimes \F', \G\boxtimes \G') \leq \max\{\gamma_h(\F, \G), \gamma_{h'}(\F', \G')\} .
  \end{equation}
\end{lem}
\begin{proof}
By \cite[Prop. 5.4.4]{K-S}, we have $\F\boxtimes \F' \in \Sh_{h+h'\geq 0}(M\times M')$. By Lemma~\ref{lem:kernel product} we also have $\K_{\varphi_{h+h'}}^a\simeq \K_{\varphi_{h}}^a\boxtimes \K_{\varphi_{h}}^a$. We deduce $\tau_{0,a}^{h+h'}(\F\boxtimes \F')\simeq \tau_{0,a}^{h}(\F)\boxtimes\tau_{0,a}^{h'}(\F')$ for $a\geq 0$.

Let $d > \max\{\gamma_h(\F, \G), \gamma_{h'}(\F', \G')\} $ and let $u,v$ (resp. $u',v'$) be interleaving morphisms for $\F$, $\G$ (resp.  $\F'$, $\G'$) for values $a,b$ with $a+b\leq d$.  Then $u\boxtimes u'$ and $v\boxtimes v'$ are interleaving morphisms for $\F\boxtimes \F'$, $\G\boxtimes \G'$ and the values $a,b$. The result follows.
\end{proof}

Let $i\colon S \to M$ be the inclusion of a closed submanifold.  We denote the transpose derivative by $(di)^t\colon S \times_M T^*M \to T^*S$.  
\begin{lem}\label{lem:kernels restriction}
  Let $h\colon T^*M\setminus 0_M \to \Real$, $h'\colon T^*S\setminus 0_{S} \to \Real$ be smooth positively homogeneous functions.  We assume that $h|_{S \times_M T^*M} = h' \circ (di)^t$. Then $\K_{h'}\simeq \K_{h}|_{S^2 \times \Real}$.
\end{lem}

\begin{proof}
  We first check that the flow $\varphi_h$ preserves the coisotropic submanifold $S \times_M T^*M$, whose symplectic reduction is $T^*S$, and induces $\varphi_{h'}$ on $T^*S$.  We take coordinates $(s,s')$ on $M$ (and $(s,s', \xi,\xi')$ on $T^*M$) such that $S = \{s'=0\}$. The hypothesis says that $h(s,0, \xi, \xi') = h'(s,\xi)$.  Since $h$ is homogeneous, we have $h(s,s', \xi,\xi') = \sum_i \xi_i (\partial h / \partial \xi_i) + \sum_j \xi'_j (\partial h / \partial \xi'_j)$.  Hence the hypothesis implies $(\partial h / \partial \xi'_j) (s,0, \xi, \xi') = 0$, for any $j$, and the Hamiltonian vector field of $h$ is tangent to $S\times_M T^*M$. More precisely $X_h|_{S\times_M T^*M} = X_{h'} + \sum_j (\partial h / \partial s'_j) \partial_{\xi'_j}$, and then, $\varphi_h^a(s,0, \xi, \xi') = (s_1,0, \xi_1, \xi'_1)$ with $(s_1, \xi_1) = \varphi_{h'}^a(s, \xi)$, as claimed.

  We want to understand the inverse image of $\K_h$ by the inclusion $j = i\times i \times \id_\Real \colon S^2\times \Real \to M^2\times \Real$. We recall that there exists an easy bound for the microsupport of the inverse image of a sheaf $\F$ by a map $f\colon X \to Y$ if $f$ is non-characteristic for $SS(\F)$, which means, when $f$ is an inclusion, that $SS(\F) \setminus 0_Y$ does not meet $T^*_XY$: in this case we have $SS(f^{-1}(\F)) \subset (df)^t(SS(\F) \cap X \times_Y T^*Y)$. When $\F$ is a direct image by $f$, $\F = f_*(\F')$ for some $\F'$, then we are in general in a characteristic case, although $f^{-1}(\F) \simeq \F'$ is easy to understand. We are in a similar case here and $j$ is a priori characteristic for $SS(\K_h)$. We decompose $j = j_2 \circ j_1$ with $j_1\colon S^2\times \Real \to M\times S \times \Real$, $j_2\colon M\times S\times \Real \to M^2 \times \Real$ the inclusions. We set
\begin{align*}
  C &=  (M\times S\times \Real) \times_{M^2 \times \Real} T^*(M^2 \times \Real)  \\
  \Gamma &= (dj_2)^t(\Gamma_{\varphi_h} \cap C) \\
  C' &=  (S^2 \times \Real) \times_{M \times S \times \Real} T^*(M \times S \times \Real) .
\end{align*}
Since $\varphi_h$ preserves $S \times_M T^*M$ and $\varphi_h^a$ is a bijection, we obtain $\Gamma \subset C'$ and $\Gamma = ((dj_1)^t)^{-1}(\Gamma_{\varphi_{h'}})$. Now $\Gamma_{\varphi_h}$ is non-characteristic for $j_2$ and we have $SS(j_2^{-1}(\K_h)) \subset \Gamma$ by~\cite[Prop. 5.4.13]{K-S}. Hence the microsupport of $j_2^{-1}(\K_h)$ is contained in the zero-section away from $S^2 \times \Real$ and it follows that $j_2^{-1}(\K_h)$ is locally constant there. Since at time $0$ $\K_h \simeq k_{\Delta_M}$ we deduce that the support of $j_2^{-1}(\K_h)$ is contained in $S^2\times \Real$, hence $j_2^{-1}(\K_h) \simeq j_{1,*}(j^{-1}(\K_h))$. By~\cite[Prop. 5.4.4]{K-S} we deduce that $SS(j^{-1}(\K_h)) \subset \Gamma_{\varphi_{h'}}$. Hence $j^{-1}(\K_h)$ satisfies the conditions which determine $\K_{h'}$, as required.
\end{proof}

\begin{lem}\label{lem:kernels restriction C0}
  Let $h \in \C^0_{ph}(T^*M)$, $h' \in \C^0_{ph}(T^*S)$ be continuous positively homogeneous functions such that $h|_{S \times_M T^*M} = h' \circ (di)^t$. Then $\K_{h'}\simeq \K_{h}|_{S^2 \times \Real}$.
\end{lem}
\begin{proof}
  It is enough to prove the result when we restrict to $S^2 \times I_\pm$ with $I_+ = [0,\infty)$ and $I_- = (-\infty,0]$. In both cases we can use the description of $\K_{h}$ as a colimit explained in~\eqref{eq:def_Kh_continuous_colimit}: $\K_h|_{M^2 \times I_+} = \colim_h \K_{h_1}|_{M^2 \times I_+}$ where $h_1$ runs over $\C^\infty_{ph, < h}(T^*M)$ (and the similar formula for $I_-$). We let $\C^\infty_{ph, < h, < h'}(T^*M) \subset \C^\infty_{ph, < h}(T^*M)$ be the subposet formed by the $h_1$ such that $h_1|_{S \times_M T^*M} = h'_1 \circ (di)^t$ for some function $h'_1 \in \C^\infty_{ph, < h'}(T^*S)$.  We can see that $\C^\infty_{ph, < h, < h'}(T^*M)$ is cofinal in $\C^\infty_{ph, < h}(T^*M)$ and the result follows from Lemma~\ref{lem:kernels restriction} and the fact that the inverse image of sheaves commutes with colimits.
\end{proof}

\begin{lem}\label{lem:distance restriction}
  Let $h \in \C^0_{ph}(T^*M)$, $h' \in \C^0_{ph}(T^*S)$ be continuous positively homogeneous functions such that $h|_{S \times_M T^*M} = h' \circ (di)^t$. Then for any $\F, \G \in \Sh(M)$ we have $\gamma_{h'}(\F|_{S}, \G|_{S}) \leq \gamma_h(\F, \G)$.
\end{lem}
\begin{proof}
  The proof is the same as the proof of Lemma~\ref{lem:bound product}, using Lemma~\ref{lem:kernels restriction C0} instead of Lem\-ma~\ref{lem:kernel product}.
\end{proof}

Now we compare two distances $\gamma_{h_1}$, $\gamma_{h_2}$ assuming $h_1\leq h_2$ on some given conic subset $W$.  We have to be careful that this inequality $h_1|_W\leq h_2|_W$ holds for any time and a natural hypothesis is to assume moreover that $W$ is stable by $\varphi_{h_1}$ and $\varphi_{h_2}$.  However, when $h_1, h_2$ are only continuous, we cannot consider their flows. Instead we assume that $W$ is given by $W = \{f \geq 0\}$ for some smooth function $f$.  For $h$ smooth, $\varphi_h$ preserves $W$ if $X_h$ is tangent to $\partial W$, that is, $\langle X_h, df \rangle|_{\partial W} =0$, or equivalently, $\langle X_f, dh \rangle|_{\partial W} =0$. The latter condition means that $h$ is constant along the trajectories of $X_f$ contained in $\partial W$, which is meaningful also when $h$ is only continuous. Moreover if a continuous $h$ satisfies this condition, it can be approximated by smooth functions also satisfying the same condition.

\begin{lem}\label{lem:compare_isot}
  Let $W \subset T^*M \setminus 0_M$, be a closed conic subset.  We assume that $W = \bigcap_{i=1}^k W_i$ with $W_i = \{f_i \geq 0\}$ for smooth functions $f_i$ on $T^*M\setminus 0_M$.  Let $h_1, h_2\colon T^*M\setminus 0_M \to \Real$ be two continuous homogeneous functions.  We assume that
  \begin{enumerate}
  \item \label{it:compare_isot1} $h_1|_W \leq h_2|_W$, 
  \item  \label{it:compare_isot2} $h_1$ and $h_2$ are constant along the orbits of the flow $\varphi_{f_i}$ contained in $\partial W$, for $i=1,\ldots, k$.
  \end{enumerate}
  Let $i_W \colon \Sh_W(M) \to \Sh(M)$ be the embedding.  Then there exist morphisms of functors $\sigma_s \colon \K_{h_1}^s \circ i_W \to \K_{h_2}^s \circ i_W$ for $s\geq 0$ such that $\tau_{s}^{h_2} = \sigma_s \circ \tau_{s}^{h_1}$.
\end{lem}
\begin{proof}
  We first assume that $h_1$ and $h_2$ are smooth and $h_1|_W < h_2|_W$.  We define an isotopy $\psi$ by $\psi_s =\varphi_{h_1,s}^{-1} \circ \varphi_{h_2, s}$. Its Hamiltonian function $h$ is given by $h_s(z) = h_2(\varphi_{h_1, s}(z)) - h_1(\varphi_{h_1, s}(z))$.  Since $\varphi_{h_1}, \varphi_{h_2}$ preserve $W$, we have $h > 0$ on $W$.  We can find a conic neighborhood $\Omega$ of $W\times\Real$ in $T^*M \times \Real$ such that $h>0$ on $\Omega$.

  Let $b\colon (T^*M \setminus 0_M)\times \Real \to {}[0,1]$ be a bump function such that $b(x,\lambda p,s) = b(x,p,s)$ for any $\lambda>0$ and $b=1$ on $W \times \Real$, $b=0$ outside $\Omega$. We set $h' = bh$ and let $\psi'$ be its Hamiltonian flow. Then $\psi'$ is nonnegative and coincides with $\psi$ on $W\times \Real$.  We deduce morphisms of functors $\id \to \K_{\psi'_s}$ for $s\geq 0$, hence $i_W \to \K_{\psi'_s} \circ i_W = \K_{\psi_s} \circ i_W$.  The compatibility with the natural morphism $\tau$ follows from the more general fact (recalled above) that the morphisms $\K_{\varphi_{h}} \to \K_{\varphi_h'}$ organize into a functor from $\C^\infty_{ph}(T^*M)$ to $\Sh(M^2\times [0,\infty))$.  Now, composing with $\K_{\varphi_{h_1, s}}$ we obtain the lemma for smooth functions.

  \smallskip

  If $h_1$ and $h_2$ are only continuous and satisfy~\eqref{it:compare_isot1} and~\eqref{it:compare_isot2}, we can find increasing sequences of smooth functions $(h_{1,n})_n$, $(h_{2,n})_n$ converging to $h_1$, $h_2$ and satisfying $h_{1,n}|_W < h_{2,n}|_W$ and~\eqref{it:compare_isot2}.  We then have morphisms $\K_{\varphi_{h_{1,n}}}^s \circ i_W \to \K_{\varphi_{h_{2,n}}}^{s} \circ i_W$. Taking the colimit over $n$ we obtain the result by~\eqref{eq:def_Kh_continuous_colimit}.
\end{proof}

\begin{prop}\label{prop:compare_dist}
  We use the notations of Lemma~\ref{lem:compare_isot} and make the same hypotheses.  Moreover we assume $h_1\geq 0$ on $W$.  Then, for $\F, \G \in \Sh_W(M)$ we have $\gamma_{h_2}(\F, \G) \leq \, \gamma_{h_1}(\F, \G)$.
\end{prop}
\begin{proof}
  Let $d \geq \gamma_{h_1}(\F, \G)$. There exist $a,b \geq 0$ such that $a+b \leq d$ and morphisms $u \colon \F \to \K_{h_1}^a(\G)$, $v \colon \G \to \K_{h_1}^b(\F)$ such that $u\circ v$ and $u\circ v$ are homotopic to the morphisms $\tau_{h_1}$.  Using the morphisms $\sigma_s$ of Lemma~\ref{lem:compare_isot} we set $u' = \sigma_a \circ u$, $v' = \sigma_b \circ v$.  We have the diagram
\[
\xymatrix{
  \F \ar[d] \ar[r]^-u & \K_{h_1}^a(\G) \ar[d]  \ar[r]^-v &  \K_{h_1}^{a+b}(\F) \ar[d] \ar[r]^-u
   &  \K_{h_1}^{2a+b}(\G) \ar[d] \\
   \F  \ar[r]^-{u'} & \K_{h_2}^{a}(\G) \ar[r]^-{v'} &  \K_{h_2}^{a+b}(\F)  \ar[r]^-{u'}
   &  \K_{h_2}^{2a+b}(\G) 
}
\]  
where the horizontal compositions of two consecutive arrows are the morphisms $\tau_{h_1}$ and $\tau_{h_2}$. We deduce that $\gamma_{h_2}(\F, \G) \leq d$, as required.
\end{proof}

\begin{lem}\label{lem: distance point ball}
  We choose a Riemannian metric $\Vert \bullet \Vert$ on $M$ and we let $g\colon T^*M\setminus 0_M \to \Real$ be the norm function $(x,p) \mapsto \|p\|$.  Let $x \in M$ and let $Z \subset M$ be a closed contractible subset. We assume that $x\in Z$ and $Z \subset B_\varepsilon$, where $B_\varepsilon$ is the open ball with center $\{x\}$ and radius $\varepsilon$, for some $\varepsilon < r_{inj}(M) / 3$.  Then $\gamma_g(k_{\{x\}}, k_Z) \leq 2\varepsilon$.
\end{lem}
\begin{proof}
  (i) Let us set $\F_s = \K^s_g\circ k_{\{x\}}$, $\G_s = \K^s_g\circ k_Z$.  For $0<r< r_{inj}(M)$ we have $\F_{-r} \simeq k_{\overline B_r}$ and $\F_r \simeq k_{B_r}[n]$, where $n = \dim M$. Indeed we know that the microsupport of $\F_r$ is the outer conormal of $\partial B_r$ outside the zero section and that $\F_r$ is supported in $\overline B_r$. It follows that $\F_r$ must be of the form $\F_r \simeq E_{B_r}$ for some $E\in \Der(k)$. Since the global sections are independent of $r$, 
  we find $k \simeq \Gamma(M; E_{B_r}) \simeq E[-n]$ and the result follows.

  The hypotheses $x\in Z$ and $Z \subset B_\varepsilon$ imply the existence of natural restriction morphisms $\F_{-\varepsilon} \to[a] \G_0 \to[b] \F_0$ and the composition $b\circ a$ is the morphism $\tau_{\varepsilon}(\F_{-\varepsilon})$.  We set $a' = \K^\varepsilon_g( a)$ and $a'' = \tau_\varepsilon(\G_\varepsilon) \circ a' \colon \F_0 \to \G_{2\varepsilon}$. Then $\K^{2\varepsilon}_g(b) \circ a'' = \tau_{2\varepsilon}(\F_0)$ and it remains to prove that $a'' \circ b = \tau_{2\varepsilon}(\G_0)$.

\medskip\noindent   (ii) Let us prove that $\Hom(\G_0, \G_{2\varepsilon}) \simeq k$.  The flow $\phi_{g, 2\varepsilon}$ sends any point $(x,p)$ with $x\in B_\varepsilon$ to a point $(x',p')$ such that $x'$ is at distance $2\varepsilon$ of $x$ in the direction $p$.  In particular $x' \in B_{3\varepsilon} \setminus B_\varepsilon$.  It follows that $SS(\G_{2\varepsilon}|_{B_\varepsilon})$ is contained in the zero section, hence $\G_{2\varepsilon}$ is constant on $B_\varepsilon$,
   say $\G_{2\varepsilon}|_{\varepsilon} \simeq E_{B_{\varepsilon}}$ for some $E\in \Der(k)$.  Now $k \simeq \Hom(\F_{-\varepsilon}, \G_0) \simeq \Hom(\F_\varepsilon, \G_{2\varepsilon}) \simeq \Hom(k_{B_\varepsilon}[n], E_{B_\varepsilon})$ and we deduce $E\simeq k[n]$. Finally, $\Hom(\G_0, \G_{2\varepsilon}) \simeq \Hom(k_Z, k_{B_\varepsilon}[n]) \simeq k$, as claimed, because $\supp(\G_0) = Z$ is contained in $B_\varepsilon$.

   Since $\Gamma(M; \G_0) \simeq \Gamma(M; \G_{2\varepsilon}) \simeq k$, we deduce that the global section morphism $u \mapsto \Gamma(M; u)$, from $\Hom(\G_0, \G_{2\varepsilon})$ to $\Hom(\Gamma(M; \G_0), \Gamma(M; \G_{2\varepsilon}))$, is an isomorphism. In particular $u$ is determined by $\Gamma(M; u)$.

   \medskip\noindent (iii) We recall that $\Gamma(M; \G_s)$ is independent of $s$ and moreover, for $c\leq d$, the natural isomorphism $\Gamma(M; \G_c) \isoto \Gamma(M; \G_d)$ is induced by $\tau_{c,d}(\G)$ (see~\cite[Prop. 4.8]{G-K-S}). It follows that $\Gamma(M; \tau_\varepsilon(\G_\varepsilon)) = \id_k$.  The same holds for $a$, hence $a'$, and for $b$.  Finally $\Gamma(M; a'' \circ b) = \id_k$. By~(ii) we conclude that $a'' \circ b = \tau_{2\varepsilon}(\G_0)$ and by~(i) this finishes the proof.
\end{proof}

\subsection{Main example: Small wrapping of the diagonal}

We consider the case $M = N\times\Real$ and put a metric on $N$.  We denote by $(x,p) \mapsto \| p\|$ the induced norm on $T^*N$.  We pick some radius $r$ and consider the function $h\colon T^*(N\times\Real) \setminus 0_{N\times\Real} \to \Real$ which is the homogenization of $\max\{0, \|p\| - r\}$.  We want to compute the sheaf $\K_h \in \Sh((N\times \Real)^2 \times \Real)$ associated with $h$. Since we work mainly in the Tamarkin category we only need to consider its composition with the projector to the Tamarkin category, that is, $\K_h^{-s} \circ k_{\Delta_N \times \{t' \leq t\}}$.  For this it is enough to know $h$ on $\{\tau\geq 0\}$.

\begin{lem}\label{lem:kernel_Hamiltonian=norm0}
  Let $f\colon \Real \to \Real$ be a smooth function.  We assume that there exist $0<a<b$ and $r\in \Real$ such that $f(u) = 0$ on $[0,a]$, $f(u) = u - r$ on $[b, \infty)$, and $f'$ gives a bijection from $[a,b]$ to $[0,1]$.  We let $h\colon T^*(N\times\Real) \setminus 0_{N\times\Real} \to \Real$ be the homogenization of $(x,p) \mapsto f(\|p\|)$, that is, $h(x,t,p,\tau) = \tau f(\|p / \tau \|)$ for $\tau > 0$, extended smoothly to $T^*(N\times\Real) \setminus 0_{N\times\Real}$.  Let $s>0$ be less than the injectivity radius of $N$. Then
  \[
  \K_h^{-s} \circ k_{\Delta_N \times \{t' \leq t\}} \simeq k_{C_s},
  \]
  where $C_s = \{(x,x',t, t') \mid d(x, x') \leq s, \; t \geq t' - s g(d(x,x')/s)\}$ and  $g(v) = f((f')^{-1}(v)) - v  (f')^{-1}(v)$.
\end{lem}
\begin{proof}
  We choose the following convenient way to extend $h$ smoothly to all of $T^*(N\times \Real)$: we set $h(x,t,p,\tau) = \|p\| - r \tau$ near $\{\tau = 0\}$ and $h(x,t,p,\tau) = - \tau f(\|p / \tau \|) - 2r\tau$ for $\tau < 0$.  (If we extend $h$ symmetrically by $h(x,t,p,\tau) = |\tau| f(\|p / \tau \|)$, we find $h(x,t,p,\tau) = \|p\| - r |\tau|$ near $\{\tau = 0\}$, which is continuous but not smooth).
  
  Let $n$ be the norm function on $T^*N \setminus 0_N$, $n(x,p) = \|p\|$. Let $X_n$ be its Hamiltonian vector field and $\varphi_n$ its flow (the normalized geodesic flow on $T^*N \setminus 0_N$). For $\tau>0$ we have $h = \tau f(\|p / \tau \|)$, hence $X_h = f'(\|p / \tau \|) X_n + f_1(\|p / \tau \|) \partial_t$, with $f_1(u) = f(u) - uf'(u)$.  Since $\|p\|$ is preserved by $\varphi_n$ and $\tau$ by the flow of $\partial_t$, we see that the coefficients of $X_n$ and $\partial_t$ are constant along the trajectories of $X_h$. We deduce $\varphi_h(x,t,p,\tau, -s) = (x',t',p', \tau)$ with $(x',p') = \varphi_n(x,p, -sf'(\|p / \tau \|))$ and $t' = t - s f_1(\|p / \tau \|)$.  In particular $d(x,x') = sf'(\|p / \tau \|)$. For $d(x,x') \in (0,s)$ we thus obtain $\|p / \tau \| = (f')^{-1}(d(x,x') / s)$, hence $t' = t - sg(d(x,x') / s)$.

  For $\tau < 0$ the computation is the same (up to changing some signs and adding $X_{-2r\tau} = -2r \partial_t$), which yields $\varphi_h(x,t,p,\tau, -s) = (x',t',p', \tau)$ with $(x',p') = \varphi_n(x,p, -sf'(\|p / \tau \|))$ and $t' = t + s(2r + g(d(x,x') / s))$.  We deduce that the front projection of the graph of $\varphi_h^{-s}$ is $\Gamma_- \cup \Gamma_+$ with
\begin{align*}
  \Gamma_- &= \{(x,x',t, t') \mid d(x, x') \leq s, \; t' = t - sg(d(x,x')/s)\} , \\
  \Gamma_+ &= \{(x,x',t, t') \mid d(x, x') \leq s, \; t' = t +s(2r + g(d(x,x')/s)) \} .
\end{align*}
Since $\K_h$ is uniquely determined by the graph of $\varphi_h$ and the condition $\K_h|_{\{s=0\}} = k_{\Delta_N}$, and taking into account the fact that  $SS(\K_h^s)$ is the graph of $\varphi_h^{-s}$, we deduce $\K_h^{-s} \simeq k_{\Gamma_0}$ with  
\begin{equation*}
  \Gamma_0 = \{(x,x',t, t') \mid d(x, x') \leq s, 
  t' - sg(d(x,x')/s) \leq t \leq t' + s(2r + g(d(x,x')/s)) \}.
\end{equation*}
Now composing with $k_{\Delta_N \times \{t' \leq t\}}$ yields the result.
\end{proof}

\begin{lem}\label{lem:kernel_Hamiltonian=norm}
  Let $r\geq 0$ be given and let $f$ be the function $f(u) = \max \{0, u - r\}$.  We set $h(x,t$, $p$, $\tau) = |\tau| f(\|p / \tau \|)$ for $\tau \not= 0$ and extend by continuity, so $h(x,t,p,\tau) = \max \{0, \|p\| - r |\tau|\}$. Let $s>0$ be less than the injectivity radius of $N$. Then 
  \[
  \K_h^{-s} \circ k_{\Delta_N \times \{t' \leq t\}} \simeq k_{C_s}
  \]
with $C_s = \{(x,x',t, t') \mid d(x, x') \leq s, \; t \geq t' + r d(x,x')\}$. 
\end{lem}
\begin{proof}
  We choose a sequence $(f_n)_n$ of differentiable functions $C^0$ converging to $f$.  We assume that the $f_n$'s satisfy the hypotheses of Lemma~\ref{lem:kernel_Hamiltonian=norm0}, that is, $f_n$ coincides with $f$ on the set $\Real_{\geq 0} \setminus I_n$, where $I_n = [r- \frac{1}{n}, r+ \frac{1}{n}]$, and $f'_n$ is increasing on $I_n$.  We set $h_n(x,t,p,\tau) = \tau f_n(\|p/\tau \|)$ for $\tau>0$. We extend $h_n$ smoothly to $T^*(N\times\Real) \setminus 0_{N\times\Real}$. By Lemma~\ref{lem:kernel_Hamiltonian=norm0} we have $\K_{h_n}^{-s} \circ k_{\Delta_N \times \{t' \leq t\}} \simeq k_{C_{n,s}}$ with $C_{n,s} = \{(x,x',t, t') \mid d(x, x') \leq s, \; t \geq t' - s g_n(d(x,x')/s)\}$ and, for $v\in [0,1]$, $g_n(v) = f_n((f'_n)^{-1}(v)) - v(f'_n)^{-1}(v)$ (recall that $f'_n$ identifies $I_n$ and $[0,1]$).  Since $(f'_n)^{-1}(v) \in I_n$, we have $0 \leq f_n((f'_n)^{-1}(v)) \leq 1/n$ and $g_n$ tends to the function $g(v) = -r v$. Hence $C_{n,s}$ tends to the set $C_s$ of the lemma.
\end{proof}

\section{Proof of the main result}

In this section we prove Theorem~\ref{thm:density sheaves}.  Since it deals with the objects $\W_{(x,a)}$, images of $k_{\{x\} \times [a, \infty)}$ by the projector $P'_{DT^*N}$, we first recall some results on $P'_{DT^*N}$.

\subsection{Some results on the projector}

We recall the following description of the projector taken from~\cite{Kuo-Shende-Zhang-Hochschild}.  Recall that for any closed subset $Z \subset T^*N$, there exists a projector $P'_Z$, right adjoint to the embedding $i_Z$ of $\Tam_Z(N)$ in $\Tam(N)$ (see the discussion before formula \ref{adjunction-formula}). Moreover $P'_Z$ is given by a convolution functor described as follows.

\begin{prop}[Prop 6.5 of \cite{Kuo-Shende-Zhang-Hochschild}]
  Let $Z\subset T^*N$ be a closed set, and let $H_n$, $n\in \N$, be any increasing sequence of compactly supported Hamiltonians supported on $T^*N \setminus Z$ such that $H_n(u) \to \infty$ for all $u\not\in Z$. Let $\K_n=\K_{H_n}^{-1}$ be their sheaf quantizations, which form an inverse system along continuation maps $ \K_n\to k_{\Delta_M \times [0,\infty)}$. Then we have $P'_Z(\F)= \lim_n \K_n (\F)$ and the counit $P'_Z \Rightarrow \id$ is intertwined with the limit of continuation maps.
\end{prop}

\begin{rem}In the proof of \cite[Prop 6.5]{Kuo-Shende-Zhang-Hochschild}, the compactly support condition was made for two purposes: 1) To guarantee the existence of sheaf quantization $\K_n$ in the great generality. 2) To make sure the Hamiltonian sequence is cofinal under the assumption $H_n(u) \to \infty$ for all $u\not\in Z$. So, for any cofinal Hamiltonian sequence $H_n$ (that may not to be compactly supported) such that $H_n(u) \to \infty$ for all $u\not\in Z$ and whose sheaf quantizations $\K_n$ exist, we can conclude the same formula as \cite[Prop 6.5]{Kuo-Shende-Zhang-Hochschild}. 

In particular, we can construct the projector $P'_Z$ for $Z=DT^*N$ in the following way. We take the continuous function $H(x,p)= \max\{0,\|p\|-1\}$ and $h$ is the homogenization of $H$, i.e. the function $h$ in Lemma~\ref{lem:kernel_Hamiltonian=norm}. Then for $H_n=\alpha_n H$ where $\alpha_n \nearrow\infty$, we take $\K_n\coloneqq \K_{h}^{-\alpha_n} \circ k_{\Delta_N \times \{t' \leq t\}}$ as their sheaf quantization. Then we have $P'_Z(\F)= \lim_n \K_n (\F)$ as well.
\end{rem}

We remark that the characterization of $P'_Z$ as a right adjoint of $i_Z$ implies that it is uniquely determined by $Z$.
\begin{lem}\label{lem:P=PK}
  Let $Z \subset T^*N$ be a closed subset.  Let $\phi\colon T^*N \to T^*N$ be a Hamiltonian isotopy which admits a contact lift $\widetilde\phi$ on $J^1(N)$ (or equivalently a lift to $T^*(N\times \Real) \setminus 0_{N\times\Real}$ as a homogeneous Hamiltonian isotopy).  Let $\K_{\widetilde\phi}$ be the sheaf associated with $\widetilde\phi$.  We assume that $\phi(Z) = Z$ and $\widetilde\phi|_{\rho^{-1}(Z)} = \id_{\rho^{-1}(Z)}$.  Then
\[
P'_Z \simeq  P'_Z \circ  \K_{\widetilde\phi}  .
\]
\end{lem}
\begin{proof}
  The hypothesis implies that $i_Z \simeq \K_{\widetilde\phi}^{-1} \circ i_Z$ and the result follows by adjunction.
\end{proof}

\begin{lem}\label{lem:formula_Wxa}
 Given a metric on $N$ and pick $s>0$ to be smaller than the injectivity radius.  For $(x_0, a_0) \in N\times\Real$ we define a closed subset of $N\times\Real$, $C(x_0, a_0, s) = \{(x,t) \mid d(x, x_0) \leq s, \, t \geq a_0 + d(x,x_0) \}$.  Then
  \begin{equation}\label{eq:formula_Wxa}
    \W_{(x_0,a_0)} \simeq P'_{DT^*N} ( k_{C(x_0, a_0, s)} ) .
  \end{equation}  
\end{lem}

\begin{figure}[htbp]
    \centering
     \begin{tikzpicture}[scale=1.0, xscale=1.2]

\def\xo{1.2}   %
\def\ao{0.8}   %
\def\s{1.3}    %
\def\Tmax{4.2} %
\def\xmin{-1.8}
\def\xmax{4.8}

\draw[->] (\xmin - 0.7,0) -- (\xmax,0) node[below] {$x$};
\draw[->] (\xmin+0.2,-0.2) -- (\xmin+0.2,\Tmax) node[left] {$t$};

\pgfmathsetmacro{\xL}{\xo-\s}
\pgfmathsetmacro{\xR}{\xo+\s}

\pgfmathsetmacro{\tL}{\ao+\s} %
\pgfmathsetmacro{\tR}{\ao+\s} %

\fill[gray!20]
(\xL,\Tmax) --
(\xR,\Tmax) --
(\xR,\tR) --
(\xo,\ao) --
(\xL,\tL) -- cycle;

\draw[thick] (\xo,\ao) -- (\xR,\tR);
\draw[thick] (\xo,\ao) -- (\xL,\tL);

\draw[dashed] (\xL,0) -- (\xL,\Tmax);
\draw[dashed] (\xR,0) -- (\xR,\Tmax);
\draw[thick] (\xL,\tR) -- (\xL,\Tmax);
\draw[thick] (\xR,\tR) -- (\xR,\Tmax);

\draw[dashed] (\xo,0) -- (\xo,\Tmax);
\draw[<->,>=Stealth]
  (\xo,\tR+1) -- node[fill=gray!20,inner sep=1mm,midway] {$s$} (\xR,\tR+1);

\fill (\xo,\ao) circle (1.2pt) node[below left] {$(x_0,a_0)$};

\draw (\xo,0) -- (\xo,-0.08) node[below] {$x_0$};

\draw (\xmin+0.2,\ao) -- (\xmin+0.12,\ao) ;
\node[left] at (\xmin+0.2,\ao) {$a_0$};
\draw (\xmin+0.2,\tR) -- (\xmin+0.12,\tR) ;
\node[left] at (\xmin+0.2,\tR) {$a_0+s$};

\node[align=left] at (\xmax,\Tmax-0.7)
{$d(x,x_0)\leq t-a_0$\\
$d(x,x_0)\leq s$};

\end{tikzpicture}
    \caption{The graph of $k_{C(x_0, a_0, s)}$.}
    \label{fig: graph of metric cone}
\end{figure}

\begin{proof}
  Let $r>1$.  We use the notations of Lemma~\ref{lem:kernel_Hamiltonian=norm}: $f$ and $h$ are defined by $f(u) = \max \{0, u - r\}$ and $h(x,t,p,\tau) = \tau f(\|p / \tau \|)$.  Let $(f_n)_n$ be a sequence of differentiable functions $C^0$ converging to $f$ and set $h_n(x,t,p,\tau) = \tau f_n(\|p\|/\tau)$.  We assume that $f_n|_{[0,1]} = 0$.  By Lemma~\ref{lem:P=PK} we have $P'_Z \simeq P'_Z \circ \K_{h_n}^s$ for all $n$ and $s$, hence $P'_Z \simeq P'_Z \circ \K_h^s$.  Lemma~\ref{lem:kernel_Hamiltonian=norm} says that $\K_h^{-s} \circ k_{\Delta_N \times [0,\infty)} \simeq k_{C_s}$ for $s>0$ and for some ``conic'' set $C_s$.  Now, $k_{C_s} \circ k_{\{x_0\} \times [a_0, \infty)} \simeq k_{C(x_0, a_0, s, r)}$ where $C(x_0, a_0, s, r) = \{(x,t) \mid d(x, x_0) \leq s, \, t \geq a_0 + r d(x,x_0) \}$.  Hence $\W_{(x_0,a_0)} \simeq P'_{DT^*N} ( k_{C(x_0, a_0, s, r)} )$.  The sheaves $k_{C(x_0, a_0, s, r)}$, $r>1$, organize into an inverse system and taking the limit as $r$ goes to $1$ gives the result.
\end{proof}

A consequence of Lemma~\ref{lem:P=PK} is the following distance comparison.  We use again the notations of Lemma~\ref{lem:kernel_Hamiltonian=norm} (with $r=1$): $f$ and $h$ are defined by $f(u) = \max \{0, u - 1\}$ and $h(x,t,p,\tau) = |\tau| f(\|p / \tau \|)$.  We also set $h^+(x,t,p,\tau) = h(x,t,p,\tau) + |\tau|$.  Hence $h^+$ is the homogenization of the continuous Hamiltonian $(x,p) \mapsto \max \{1, \|p\|\}$.
\begin{prop}\label{prop:bound gammatau gammaK}
  Let $\F,\G \in \Tam(N)$. Then
  \[
  \gamma_\tau(P'_{DT^*N}(\F), P'_{DT^*N}(\G)) \leq \gamma_{h^+}(\F, \G) .
  \]
\end{prop}
\begin{proof}
  As in the proof of Lemma~\ref{lem:formula_Wxa} we approximate $f$ by a sequence of smooth functions $(f_n)_n$ and let $(h_n)_n$ be the corresponding Hamiltonians.  Since $h_n^+=h_n+\vert \tau\vert$ and  $h_n$ and $\tau$ commute, we have $\phi_{h^+_n,s} = \phi_{h_n, s} \circ \phi_{\tau, s} = \phi_{\tau, s} \circ \phi_{h_n, s}$. As a result  the same relation holds for their associated sheaves.  We can also see that $P'_{DT^*N}$ commutes with $\K_{\tau,s}$.  Hence, if there exists a morphism $u\colon \F \to \K_{h^+_n,a} \circ \G$, applying $P'_{DT^*N}$ and using Lemma~\ref{lem:P=PK}, we obtain $P'_{DT^*N}(u) \colon P'_{DT^*N}(\F) \to \K_{\tau,a}(P'_{DT^*N}(\G))$.  Since the definition of the distance only involves morphisms of the kind $\F \to \K_{h^+_n,a} \circ \G$ and $\G \to \K_{h^+_n,b} \circ \F$, the result follows.
\end{proof}

\begin{rem} It is clear that Proposition~\ref{prop:bound gammatau gammaK} is true for all closed $Z\subset T^*N$ and homogeneous Hamiltonians $h^+\colon  T^*(N\times \Real) \setminus 0_{N\times\Real}\rightarrow \Real$ that restrict to $h^+|_{\rho^{-1}(Z)}=\tau$ and $\F,\G\in \Tam(N)\cap \Sh_{h^+\geq 0}(N\times \Real)$, where we only take $Z=DT^*N$ and a particular $h^+$ here.
\end{rem}

Thanks to Proposition~\ref{prop:bound gammatau gammaK}, we shall estimate $\gamma_{h^+}$ for related sheaves.

\begin{lem}\label{lem:approx fibre boule}
  Let $x\in N$, $\varepsilon >0$ and let $Z\subset N$ be a closed contractible subset contained in the ball with center $x$ and radius $\varepsilon < r_{inj}(N)/3$.  Then, for any $\F \in \Sh(N\times\Real)$ such that $SS(\F) \subset \{\tau \geq \| p \| \}$, we have $\gamma_{h^+}(\F \otimes k_{\{x\} \times \Real}, \F \otimes k_{Z \times \Real}) \leq  4\varepsilon$.   
\end{lem}
\begin{proof}
  We recall that $h^+(x,p,t,\tau) = \max\{ \tau, \| p \| \}$ for $\tau \geq 0$.  Since we work on $\Sh_{\tau\geq0}(N\times\Real)$ we can as well assume $h^+(x,p,t,\tau) = \max\{ \tau, \| p \| \}$ for any $\tau$.  We first give a bound on $N^2\times \Real^2$ and then take the pull-back by the diagonal embedding $\delta \colon N\times\Real \to (N\times\Real)^2$. To distinguish, we denote points of the second copy of $T^*(N\times\Real)$ with prime, and then points of $T^*(N^2\times\Real^2)$ are denoted by $(x,p,x',p',t,\tau, t',\tau')$.  We will apply Lemma~\ref{lem:distance restriction} and we look for a function $h_1$ on $T^*(N^2\times\Real^2)$ such that $h_1|_{\Delta_1} = h^+ \circ (d\delta)^t$, where $\Delta_1 = \Delta_{N\times\Real} \times_{N^2\times\Real^2} T^*(N^2\times\Real^2)$.  We have $(d\delta)^t(x,p,x,p',t,\tau, t,\tau') = (x,t, p+p', \tau+\tau')$ so we should extend the function $\|p+p'\|$ outside the diagonal. We let $\Delta(r)$ be the neighborhood of $\Delta_N$ formed by the pairs $(x,x')$ such that $d(x,x') < r$. For $(x,x') \in \Delta(r_{inj}(N))$ we identify $T_x^*N$ and $T_{x'}^*N$ through the parallel transport along the shortest geodesic between $x$ and $x'$; we can then make sense of $p+p'$.  We choose a partition of unity $\alpha, \beta$ with $\alpha=1$ on $\Delta(r_{inj}(N)/2)$ and $\beta=1$ outside $\Delta(r_{inj}(N))$.  We set $h_0(x,p,x,p',t,\tau, t,\tau') = \alpha \|p+p'\| + \beta \|(p,p')\|$.  We have the inequality
\[
  \|p\| + h_0 \geq \|p'\|
  \]
  because over $\Delta(r_{inj}(N))$ we have $\|p\| + h_0 \geq \alpha \|p\| + \alpha \|p+p'\| + \beta \|p'\| \geq (\alpha + \beta)\|p'\| = \|p'\|$ and outside $\Delta(r_{inj}(N))$ we even have $h_0 \geq \|p'\|$.
  
\smallskip
  
Now we define $h_1$ on $T^*(N^2\times\Real^2)$ by $h_1 =\max\{\tau +\tau', h_0\}$. We bound the distance between $\F \boxtimes k_{\{x\} \times \Real}$ and $\F \boxtimes k_{Z \times \Real}$ for the distance $\gamma_{h_1}$.  We define $h_2$ on the second copy of $T^*(N\times\Real)$ by $h_2 = \| p' \|$. We remark that Lemma~\ref{lem:bound product} makes sense in the degenerated case where $h$ or $h'$ is the zero function. Writing $h_2 = 0 + 0 + \|p'\| + 0$ we deduce, by Lemma~\ref{lem: distance point ball}, that $\gamma_{h_2}(\F \boxtimes k_{\{x\} \times \Real}, \F \boxtimes k_{Z \times \Real}) \leq 2\varepsilon$.  For $\G\in \Sh_{\tau'\geq0}(N\times\Real)$ we have $SS(\F\boxtimes \G) \subset W$ with $W = \{\tau \geq \| p \| , \tau'\geq 0\}=\{\tau^2 - \| p \|^2 \geq 0,\tau\geq 0 , \tau'\geq 0\} \subset T^*(N^2\times\Real^2)$.  When we restrict to $W$, we have
\[
h_1 \geq \frac12 (\tau + \tau') + \frac12 h_0
\geq \frac12 \| p \| + \frac12 h_0
  \geq \frac12 h_2 .
\]
By Proposition~\ref{prop:compare_dist} we deduce
\[
\gamma_{h_1}(\F \boxtimes k_{\{x\} \times \Real}, \F \boxtimes k_{Z \times \Real})
  \leq 2 \gamma_{h_2}(\F \boxtimes k_{\{x\} \times \Real}, \F \boxtimes k_{Z \times \Real})
  \leq 4 \varepsilon
\]
and by Lemma~\ref{lem:distance restriction}, we have
\[
  \gamma_{h^+}(\F \otimes k_{\{x\} \times \Real}, \F \otimes k_{Z \times \Real})\leq  \gamma_{h_1}(\F \boxtimes k_{\{x\} \times \Real}, \F \boxtimes k_{Z \times \Real})\leq 4 \varepsilon,
\]
which finishes the proof.
\end{proof}

\subsection{Some lemmata on iterated cones}

Let $\C$ be a stable $\infty$-category \cite[Def. 1.1.1.9]{HigherAlgebra}\footnote{Stable $\infty$-categories are $\infty$-categories that play the role of triangulated categories \cite[Theorem 1.1.2.14]{HigherAlgebra}. Sheaf categories $\Sh(M)$ are stable.} and $\F, \G_0, \ldots, \G_n$ be objects of $\C$.  We say that $\F$ is an ($n$ steps) iterated cone on $(\G_0, \ldots, \G_n)$ if there exist objects $\F_0 = \G_0$, \dots, $\F_n = \F$ and morphisms $f_i\colon \F_i \to \G_{i+1}$ such that $\F_{i+1}$ is isomorphic to the cone of $f_i$. The next lemmas hold for any stable $\infty$-category with an $\arbR$-action, where $\gamma$ denotes the corresponding interleaving distance (instead of $\gamma_{h^+}$ and $T_s = \K_{h^+}^s \circ -$ for sheaves). 

\begin{lem}\label{lem:distance cone0}
  Let $u\colon \F' \to \F$ and $f \colon \F \to \G$ be two morphisms. We assume that there exists $v\colon \F \to T_\varepsilon(\F')$ such that $v\circ u$ and $T_\varepsilon(u) \circ v$ are the canonical morphisms $\tau$.  Let $C(f)$, $C(f\circ u)$ be the cones of $f, f\circ u$. Then $\gamma(C(f), C(f\circ u)) < 4 \varepsilon$.
\end{lem}
\begin{proof}By the octahedral axiom (now a proposition in stable categories), we have a fiber sequence $C(u) \to C(f\circ u) \to[w] C(f) \to[+1]$ such that $C(w) \simeq C(u)[1]$.  By~\cite[Lem. 6.8 (i)]{Guillermou-Viterbo} and the existence of $v$ we have $\gamma(0, C(u)) < 2\varepsilon$, hence $\gamma(0, C(w)) < 2\varepsilon$, and then we apply ~\cite[Lem. 6.8 (ii)]{Guillermou-Viterbo} to $w$ to deduce that $\gamma(C(f), C(f\circ u)) < 4 \varepsilon$.  
\end{proof}

\begin{lem}\label{lem:distance cone}
  Let $f \colon \F \to \G$ be a morphism and let $\F', \G'$ be two objects such that $\gamma(\F, \F')$, $\gamma(\G,\G') < \varepsilon$.  Then there exist $0\leq a , b < \varepsilon$ and a morphism $f' \colon T_{-a}(\F') \to T_b(\G')$ such that $\gamma(C, C') < 8\varepsilon$, where $C,C'$ are the cones of $f,f'$.
\end{lem}
\begin{proof}
  Since $\gamma(\F, \F') < \varepsilon$, there exist $0\leq a,a'$ and morphisms $u_1\colon \F \to T_{a'}(\F')$, $u_2\colon \F' \to T_a(\F)$ such that $T_{a'}(u_2) \circ u_1$ and $T_a(u_1) \circ u_2$ are homotopic to canonical morphisms $\tau_{a'+a}$ and $a'+a < \varepsilon$.  We have similar morphisms $v_1\colon \G \to T_b(\G')$, $v_2\colon \G' \to T_{b'}(\G)$.  We define $f' = v_1 \circ f \circ T_{-a}(u_2)$ and also $f_1 = f \circ T_{-a}(u_2)$.  By Lemma~\ref{lem:distance cone0} (and its variant which deals with $u\circ f$) both $\gamma(C, C(f_1))$ and $\gamma(C(f_1), C')$ are less than $4\varepsilon$.  The result follows from the triangle inequality.
\end{proof}

\begin{lem}\label{lem:approx iterated cone}
  Let $\F$ be an $n$ steps iterated cone on $\G_0, \ldots, \G_n$.  Let $\varepsilon>0$ and $\G'_0, \ldots, \G'_n$ be such that $\gamma(\G_i, \G'_i) < \varepsilon$ for all $i$.  Then there exist $a_0,\ldots, a_n$ and an $n$ steps iterated cone $\F'$ on $T_{a_0}(\G'_0), \ldots, T_{a_n}(\G'_n)$ such that $\gamma(\F, \F') < 8^n \varepsilon$.
\end{lem}
\begin{proof}
  We prove this by induction on $n$. For $n=1$ the result is given by Lemma~\ref{lem:distance cone}.  We assume the result holds for $n-1$ and we consider an $n$ steps iterated cone.  By definition there exist $\F_0 = \G_0$, \dots, $\F_n = \F$ and morphisms $f_i\colon \F_i \to \G_{i+1}$ such that $\F_{i+1}$ is isomorphic to the cone of $f_i$.  The induction step gives $a_0,\ldots, a_{n-1}$, $\F'_0 = T_{a_0}(\G'_0)$, \dots, $\F'_{n-1}$, and morphisms $f'_i\colon \F'_i \to T_{a_i}\G'_{i+1}$ such that $\F'_{i+1}$ is isomorphic to the cone of $f'_i$ and $\gamma(\F_{n-1}, \F'_{n-1}) < 8^{n-1} \varepsilon$.  We apply Lemma~\ref{lem:distance cone} to $f_{n-1} \colon \F_{n-1} \to \G_n$, $\F'_{n-1}$ and $\G'_n$, with $8^{n-1} \varepsilon$ instead of $\varepsilon$.  We obtain a morphism $f'_n\colon T_{-a}(\F'_{n-1}) \to T_b(\G'_n)$, for some $a,b$, such that $\gamma(\F'_n, \F_n) < 8^n \varepsilon$, where $\F'_n$ is the cone of $f'_n$.  We shift $a_0,\ldots, a_{n-1}$ by $-a$ and set $a_n = b$. This proves the result.
\end{proof}

\begin{lem}\label{lem:distance somme}
  Let $\varepsilon >0$ and let $\F_i$, $\G_i$, $i\in I$, be two families of objects such that $\gamma(\F_i, \G_i) \leq \varepsilon$ for any $i\in I$. We assume that $\mathcal{C}$ admits $I$-indexed direct sums. Then $\gamma(\bigoplus_{i\in I} \F_i, \bigoplus_{i\in I} \G_i) \leq 2\varepsilon$.
\end{lem}
\begin{proof}
  The hypothesis imply in particular that there exist $\varphi_i\colon \F_i \to T_\varepsilon(\G_i)$, $\psi_i\colon \G_i \to T_\varepsilon(\F_i)$ such that $\varphi_i\circ \psi_i$ and $\psi_i\circ \varphi_i$ are homotopic to canonical morphisms.  We deduce the lemma by considering $\varphi = \bigoplus_i \varphi_i$ and $\psi = \bigoplus_i \psi_i$.
\end{proof}

From now on, we consider sheaves. Using a \v Cech resolution we can express any sheaf $\F$ as an $n$ steps iterated cone on sheaves of the form $\G_i = \bigoplus_{j \in I_i} \F \otimes k_{Z_j}$, $i=0,\ldots,n$ (see Lemma~\ref{lem:Cech}). Using Lemma~\ref{lem:approx fibre boule} and previous approximation Lemmas proved in this section we can replace the $\F \otimes k_{Z_j}$ by $\F \otimes k_{\{x_j\} \times \Real}$, for some $x_j \in Z_j$, and obtain an iterated cone on the $\bigoplus_{j \in I_i} \F \otimes k_{\{x_j\} \times \Real}$, which approximate $\F$.  Now $\F \otimes k_{\{x_j\} \times \Real}$ is approximated by the iterated cones of fibers using Lemma~\ref{lem:iterated cones dim 1} and we will be able to conclude. We give more details in~\S\ref{sec:proofmainthm}.

\begin{lem}\label{lem:Cech}
  Let $M = \bigcup_{i\in I} U_i$ be a finite covering.  We set $U_J^{cl} = \bigcap_{j\in J} \overline{U_j}$ and we assume that $U_J^{cl} =\emptyset$ for $\vert J \vert \geq n+2$ and that all $U_J^{cl}$ are contractible.  Let $k_M \simeq C_0 \to C_1 \to \cdots \to C_n$ be the \v Cech resolution of the constant sheaf, where $C_i = \bigoplus_{J \subset I, |J| = i+1} k_{U_J^{cl}}$. Then for any sheaf $\F$ on $M$, we have $\F \in \langle \F \otimes C_0,\ldots, \F \otimes C_n \rangle_{\Sh(M)}$, and more precisely $\F$ is an $n$ steps iterated cone on $\F \otimes C_0[1]$, $\F \otimes C_1$, \dots, $\F \otimes C_n[1-n]$.
\end{lem}
\begin{proof}
  Let $D_i$ be the truncated complex $D_i = C_0 \to C_1 \to \cdots \to C_i$.  We have fiber sequences $D_i \to D_{i-1} \to C_i[1-i] \to[+1]$.  We deduce by induction on $i$ that $\F\otimes D_i$ is an $i$ steps iterated cone on $\F \otimes C_0,\ldots, \F \otimes C_i[1-i]$. For $i=n$ we obtain the lemma.
\end{proof}

\begin{lem}\label{lem:iterated cones dim 1}
  Let $k$ be ring such that any finite type module has a free resolution of finite length.  Then any constructible $\F \in \Tam(\pt)$ (with coefficients $k$) is an iterated cone of sheaves of the type $k_{[a, \infty)}[d]$.
\end{lem}
\begin{proof}
  We assume that $\F$ is constructible with respect to a stratification given by finitely many points $a_1<\cdots < a_N$ and the open intervals they define.  Then $\F \simeq E_{(-\infty, a_1)}$ on $(-\infty, a_1)$ for some $k$-module $E$ and $\Hom(\F, k_{(-\infty, a]}) \simeq E$ for any $a<a_1$. Since $\F \in \Tam(\pt)$, it is left orthogonal to $ k_{(-\infty, a]}$ and we obtain $E\simeq 0$.  We deduce that $\F|_I \simeq (E_1)_{[a_1, \infty)}|_I$ for some neighborhood $I$ of $\{a_1\}$ and some $k$-module $(E_1)$ (see for example~\cite[Example 1.2.3-(v)]{Guillermou-Asterisque}).  Since $SS(\F) \subset \{\tau\geq0\}$, the microlocal Morse theorem implies that the isomorphism $\F|_I \simeq (E_1)_{[a_1, \infty)}|_I$ extends to a morphism $u\colon \F \to (E_1)_{[a_1, \infty)}$.  Then the cone of $u$ is constructible with respect to the stratification induced by $a_2<\cdots < a_N$. An induction on $N$ shows that $\F$ is an iterated cone of sheaves of the form $(E_i)_{[a_i, \infty)}$.  Now the hypothesis on $k$ implies that it is also an iterated cone of sheaves of the form $k_{[a_i, \infty)}[d]$.
\end{proof}

\subsection{Proof of Theorem~\ref{thm:density sheaves}}\label{sec:proofmainthm}
Let $0<\varepsilon < r_{inj}(N)/3$.  We pick a finite covering $N = \bigcup_{i\in I} U_i$ be a finite covering as in Lemma \ref{lem:Cech} and assume moreover that all $U_i$'s are contained in a ball of radius less than $\varepsilon$.  We set $U_J^{cl} = \bigcap_{j\in J} \overline{U_j}$.  By Lemma \ref{lem:Cech} $\F$ is an $n$ steps iterated cone on $\F \otimes C_0[1]$, $\F \otimes C_1$, \dots, $\F \otimes C_n[1-n]$, where $C_i = \bigoplus_{J \subset I, |J| = i+1} k_{U_J^{cl}\times \Real}$. For each $J\subset I$ we pick $x_J$ such that $U_J^{cl} \subset B_\varepsilon(x_J )$.  Lemma \ref{lem:approx fibre boule} says that $\gamma_{h^+}(\F \otimes k_{\{x_J\} \times \Real}, \F \otimes k_{U_J^{cl} \times \Real}) \leq 4\varepsilon$.  We set $C'_i = \bigoplus_{J \subset I, |J| = i+1} k_{\{x_J\}\times \Real}$.  By Lemma~\ref{lem:distance somme} we have $\gamma_{h^+}(\F \otimes C_i, \F \otimes C'_i) \leq 8 \varepsilon$.

Now for any $J$, we set $\F_J=\F \otimes k_{\{x_J\} \times \Real} $. By our hypotheses $\F_J$ is a limit of constructible sheaves and, by Lemma~\ref{lem:iterated cones dim 1}, we deduce that there exists $D_J$, which is an iterated cone of  sheaves of the type $k_{\{x_J\} \times [a, \infty)}[d]$, such that $\gamma_{\tau}(\F_J, D_J) < \varepsilon$. Because $\tau \leq h_+$ on $\{\tau \geq 0\}$, we have $\gamma_{h^+}(\F_J, D_J)\leq \gamma_{\tau}(\F_J, D_J)<\varepsilon$ by Lemma~\ref{prop:compare_dist}.

Therefore, for any $i$, the sheaf $D_i=\bigoplus_{J \subset I, |J| = i+1} D_J   $ satisfies $\gamma_{h^+}(\F\otimes C_i', D_i)<2\varepsilon $, and then $\gamma_{h^+}(\F\otimes C_i, D_i)<10\varepsilon $ by triangle inequality. It follows from Lemma~\ref{lem:approx iterated cone} that there exist $a_0,\ldots, a_n$ and an $n$ steps iterated cone $\F'$ on $T_{a_0}(D_0), \ldots, T_{a_n}(D_n)$ such that $\gamma_{h^+}(\F, \F') <  10\times 8^{n} \varepsilon$.
 
Now we apply the projector $P'_{DT^*N}$ to our iterated cones.  Each $\G_i = P'_{DT^*N}(T_{a_i}(D_i))$ is a direct sum of sheaves of the type $P'_{DT^*N}(T_{a_i}(D_J))$, then an iterated cone of $\W_{(x_J,a_i+b)}[d]=P'_{DT^*N}(k_{\{x_J\} \times [a_i+b, \infty)}[d])$. Then $C= P'_{DT^*N}(\F')$ is an iterated cone on $\G_0,\ldots, \G_n$, and then an iterated cone of $\W_{(x_J,a)}[d]$. Since $\F \simeq P'_{DT^*N}(\F)$, Proposition~\ref{prop:bound gammatau gammaK} implies that $\gamma_\tau(\F, C)\leq \gamma_{h^+}(\F, \F') <   10\times 8^{n} \varepsilon$, which proves the result.\qed

\begin{rem}\label{rem: interleaving Rouquier dimension}
The last proof works with $k = \Z$. If we work over a field, we can replace the use of Lemma~\ref{lem:iterated cones dim 1} by the fact that any constructible sheaf is a direct sum of sheaves of the type $k_{\{x_J\} \times [a, b)}[d]$ (by Gabriel's theorem) with $b$ a real number or $b=+\infty$. The case $b< +\infty$
is given as the cone of $k_{\{x_J\} \times [a, \infty)} \to k_{\{x_J\} \times [b, \infty)}$. Then the proof shows that it is enough to take $\dim(N) + 1$ cones to get an approximation of a sheaf by the sheaves $\W_{(x,a)}$ (if we don't count taking direct sums into the total count of cones). 

\end{rem}

\appendix
\section{Interleaving Rouquier dimension}\label{sec:InterleavingRouquierdimension}

The Remark~\ref{rem: interleaving Rouquier dimension} leads to a generalization of Rouquier dimension for stable $\infty$-categories with $\arbR$-actions.

For a stable $\infty$-category $\mathcal{ C}$ equipped with an $\arbR$-action $T$, we set its interleaving distance to be $\gamma_\mathcal{C}$. We say a subset of objects $G=\{G_i:i\in I\}$ that is closed under the $\arbR$-action of $\mathcal{ C}$ is a set of \textit{solo-approximators} if: For any $i,j\in I$, we have $\gamma_\mathcal{C}(G_i,G_j)< +\infty$, and moreover, for any object $X \in \mathcal{C}$ and any $\varepsilon >0$, there exists an iterated cone $C$ out of some $G_i$ such that $\gamma_\mathcal{C}(X,C)< \varepsilon$.

For $d\geq -1$, we denote $\langle G_i\mid i\in I\rangle_d$ the smallest full thick (that is, closed under taking direct summand) subcategory of $\mathcal{ C}$ that contains all step $d+1$ iterated cones out of direct sums $\bigoplus_{i \in I_k}G_i$, $k=0,\ldots,d$, with $I_k \subset I$.

\begin{defn}The interleaving Rouquier dimension of $(\mathcal{C},T)$, denoted by $\operatorname{IRdim}(\mathcal{C},T)$ (or simply $\operatorname{IRdim}(\mathcal{C})$ if $T$ is clear), is defined as the minimal $d$ such that there exists a set of solo-approximators $G=\{G_i\mid i\in I\}$ such that for any $X \in \mathcal{ C}$ and any $ \varepsilon >0$, there exists $C\in \langle G_i\mid i\in I\rangle_d$ such that $\gamma_\mathcal{C}(X,C)< \varepsilon$.    
\end{defn}

We now explain the relation between the interleaving Rouquier dimension and the usual Rouquier dimension $\operatorname{Rdim}$\footnote{The Rouquier dimension of $\mathcal{ C}$ is defined as the minimal $d$ such that for some $G$ we have $\mathcal{ C}=\langle G\rangle_d$. It is clear that  $\operatorname{Rdim}$ only depends on the homotopy category of $ho(\mathcal{ C})$ (similarly, $\operatorname{IRdim}$ only depends on $ho(\mathcal{ C})$ and $ho(T)$.)}. For any stable $\infty$-category $\mathcal{ C}$ equipped with an $\arbR$-action, one can define a stable $\infty$-category $\mathcal{C}_\infty$ as the Verdier quotient of $\mathcal{C}$ by Tamarkin torsion objects of $\mathcal{C}$ (see~\cite[Definition 6.1]{G-S}, \cite[Definition 5.19]{APT-Kuwagaki-Zhang}).

We have
\begin{prop}\label{prop: dimension comparision}If $\mathcal{C}$ has a set of solo-approximators $G=\{G_{i}\mid i\in I\}$, then  $\mathcal{C}_\infty$ is generated by a single $G_{i}$ (for any $i\in I$), and moreover, $\operatorname{Rdim}(\mathcal{C}_\infty)\leq \operatorname{IRdim}(\mathcal{C})$.
\end{prop}
\begin{proof} Any two objects $X,Y$ in $\mathcal{C}$ with $\gamma_\mathcal{C}(X,Y)<\infty$ become isomorphic in $\mathcal{C}_\infty$. Then the finite distance condition for solo-approximators tells that all $G_{i},G_{j}$ are isomorphic to each other. Moreover, the density condition shows that any $X\in \mathcal{C}$ is isomorphic in $\mathcal{C}_\infty$ to an iterated cone out of $G_{i}$. This proves the first claim. Moreover, if $\operatorname{IRdim}(\mathcal{C})= d$, then any $X\in \mathcal{C}$ are isomorphic in $\mathcal{C}_\infty$ to a iterated cone of $G_{i}$ of step $d+1$, which means $\operatorname{Rdim}(\mathcal{C}_\infty)\leq d$. 
\end{proof}

Now we relate Remark~\ref{rem: interleaving Rouquier dimension} and the interleaving Rouquier dimension. Let $\Tam_{DT^*N}(T^*N)_{lc}$ to be the full subcategory of $\Tam_{DT^*N}(T^*N)$ consisting of objects $\F$ satisfying that for any $x\in N$ the sheaf $\F \otimes k_{\{x\} \times\Real}$ is a $\gamma^s$-limit of constructible sheaves on $\Real$. We have that $\Tam_{DT^*N}(T^*N)_{lc}$ is a stable $\infty$-category that is invariant under the $\arbR$-action of $\Tam_{DT^*N}(T^*N)$: The $\arbR$-invariance is clear; and we only need to verify the stability in the case $N=\pt$ because $-\otimes k_{\{x\} \times\Real}$ is an exact functor. The $N=\pt$ case follows from Lemma~\ref{lem:distance cone}. Consider $\{ \W_{(x,a)}\mid (x,a)\in N\times \Real=I\}$, the density theorem almost tells that it forms a set of solo-approximator of $\Tam_{DT^*N}(T^*N)_{lc}$, except that it is not clear whether the $\W_{(x,a)}$'s are indeed objects of $\Tam_{DT^*N}(T^*N)_{lc}$ (which should be, but we will not verify it here). So, we take $\mathcal{C}(DT^*N)$ as the thick stable subcategory of $\Tam_{DT^*N}(T^*N)$ spanned by $\{ \W_{(x,a)}\mid (x,a)\in N\times \Real=I\}$ and $\Tam_{DT^*N}(T^*N)_{lc}$, which is clearly $\arbR$-invariant.

\begin{prop}\label{prop: two dimension}Let $\mathcal{C}(DT^*N)$ be the stable $\infty$-category defined as above equipped with the $\arbR$-action of $\Tam_{DT^*N}(T^*N)$. Then $\{ \W_{(x,a)}\mid (x,a)\in N\times \Real=I\}$ is a set of solo-approximators of $\mathcal{C}(DT^*N)$, and if $k$ is a field, then we have $\operatorname{IRdim}(\mathcal{C}(DT^*N))\leq \dim N$.
\end{prop}
\begin{proof}It remains to check that any two $\W_{(x,a)}$ have a finite interleaving distance. In fact, by Lemmata~\ref{lem:formula_Wxa} and~\ref{lem: distance point ball}~as well as Proposition~\ref{prop:bound gammatau gammaK}, we have that if $d(x,y)<\varepsilon$ and $|a-b|<\varepsilon$ then $\gamma^s (\W_{(x,a)},\W_{(y,b)})< 2\varepsilon$. Now, we pick a path between $(x,a) $ and $ (y,b)$ in $N\times \Real$, and the finite distance follows if we subdivide the path finely enough.
\end{proof}

\begin{rem}
Here, we consider the (unfiltered) wrapped Fukaya category $\WFuk (T^*N)$ (and its perfect modules category): It is proven in ~\cite[Section 4]{Bai_Cote_Rouquier_dimension} that $\operatorname{Rdim}(\operatorname{Perf}(\WFuk (T^*N))) \leq \dim N$. 

On the other hand, we have $\operatorname{Rdim}((\mathcal{C}(DT^*N))_\infty ) \leq \dim N$ by Propositions~\ref{prop: dimension comparision} and~\ref{prop: two dimension}. 

Those two results are related in the following way: we claim that 
\[\operatorname{Idem}((\mathcal{C}(DT^*N))_\infty) \simeq \operatorname{Perf}(\WFuk (T^*N)),\]where $\operatorname{Idem}$ means idempotent completion. Then two $\operatorname{Rdim}$ bounds are the same one (notice that idempotent completion does not increase Rouquier dimension). Here is the idea of the proof: By \cite{Abouzaid-fiber-generation}, the latter is generated by a single cotangent fiber $T^*_xN$ whose endomorphism algebra is isomorphic to the chain over the based loop space. On the other hand, we already know that $(\mathcal{C}(DT^*N))_\infty$ is generated by a single $\W_{(x,a)}$ (which also a wrapped version of a single cotangent fiber), so does $\operatorname{Idem}((\mathcal{C}(DT^*N))_\infty)$, it remains to check that the endomorphism algebra of $\W_{(x,a)}$ (in $(\mathcal{C}(DT^*N))_\infty$) is also the chain over the based loop space. This computation can be done by \cite[Proposition 5.21]{APT-Kuwagaki-Zhang} and a variant of ~\cite[Thm. 4.4]{Zhang-Cyclic}.

The proof of $\operatorname{Rdim}(\operatorname{Perf}(\WFuk (T^*N))) \leq \dim N$ we give here seemingly differs from~\cite{Bai_Cote_Rouquier_dimension}, where the authors use the sectorial descent of~\cite{G-P-S-3}. However, our proof for the density theory uses a \v Cech descent argument for sheaves on the base $N$, which should be regarded as a filtered analogy of the sectorial descent of~\cite{G-P-S-3} (that is not proven yet.) We then believe both proofs are essentially the same in a certain context.
It was pointed out to us by O. Cornea that a type of complexity also appears in \cite{A-B-C}. \end{rem}

\section{Non-compactness of \texorpdfstring{$\widehat{\mathfrak L}(M,\omega)$}{L(M,w)}}
We prove the following  which --- at least for the two-dimensional  case --- is implicitly contained in \cite[Section 5]{MCA-VH-CV}
   \begin{thm}\label{Thm-non-compactness} Let $(M,\omega)$ be any symplectic manifold. 
    Then the space $(\mathfrak L(M,\omega),\gamma)$ of embedded exact Lagrangians is either empty or not totally bounded.
    As a consequence $(\DHam(M,\omega),\gamma)  $ is not totally bounded either. 
\end{thm}
\begin{rem}
\begin{enumerate}
\item As pointed out by O. Cornea, we prove slightly better: no ball in $(\mathfrak L(M,\omega),\gamma)$ is totally bounded. 
    \item 
    Theorem \ref{Thm-non-compactness} was proved in \cite{A-B-C} when $M=T^*N$ for $N$ a manifold of negative curvature.
    
\end{enumerate}
\end{rem} 

Notice that stating that $(\mathfrak L(M,\omega),\gamma)$ is not totally bounded is equivalent to stating that its completion  $(\widehat{\mathfrak L}(M,\omega),\gamma)$ is non-compact. We are going to exhibit a sequence $(L_k)_{k\geq 1}$ in $(M,\omega)$ such that $\gamma(L_k,L_l)> \varepsilon_0>0$. Such a sequence cannot have a converging subsequence, which implies our theorem. We first construct such a sequence locally, that is in a domain in $DT^*([-1,1])$. In fact such a sequence occurs naturally as we saw in \cite[Proposition 5.11]{MCA-VH-CV}, and we shall base our proof on \cite[Lemma 5.12]{MCA-VH-CV}. 

We consider an exact Lagrangian in $[-1,1]^2=DT^*([-1,1])\subset T^*([-1,1])$ coinciding with the zero section near $q=-1,1$. This is nothing but an embedded curve coinciding with $[-1,1]\times \{0\}$ near the endpoints, such that the area enclosed by the curve and $p=0$ (counted with signs) equals zero. 
\begin{figure}[ht]
 \begin{overpic}[width=7cm]{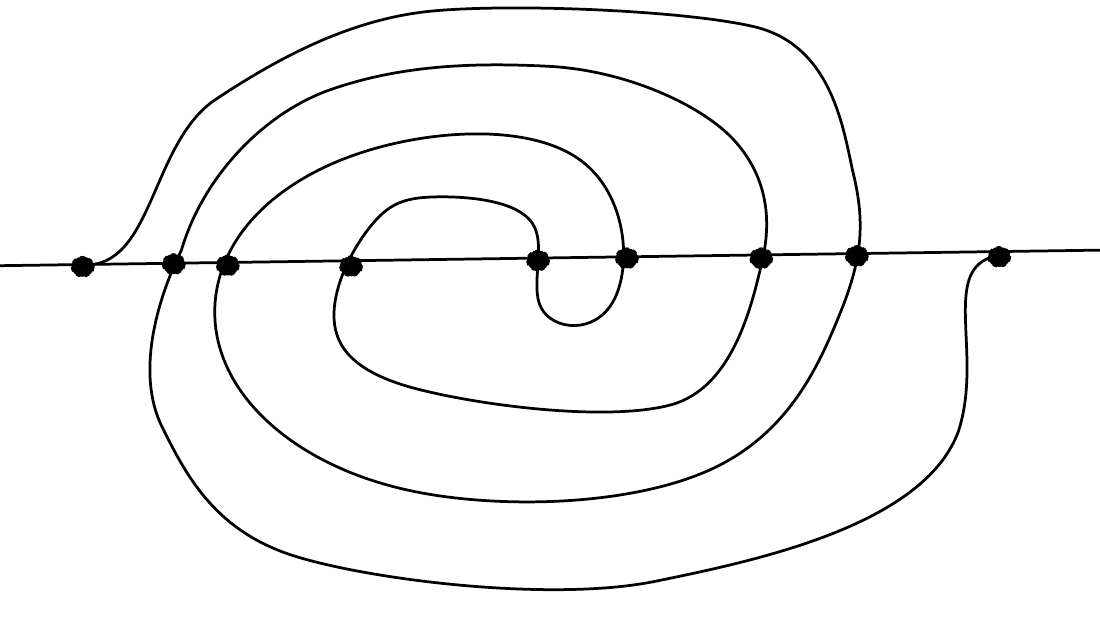}
 \put(100,32){$L_0$}
  \put(35,0){$L_1$}
   \put(92,30){$0$}
   \put(80,30){$1$}
    \put(65,30){$2$}
   \put(58,30){$3$}
   \put(50,30){$4$}
   \put(34,30){$3$}
   \put(22,29){$2$}
   \put(12,28){$1$}
   \put(8,28){$0$}  
 \end{overpic}
\caption{ $L_0,L_1$ and the Conley-Zehnder-Maslov index of the intersections}
\label{Fig-1}
\end{figure}

\begin{figure}[H]
\begin{subfigure}{0.4\textwidth}
\centering
 \begin{overpic}[width=7cm]{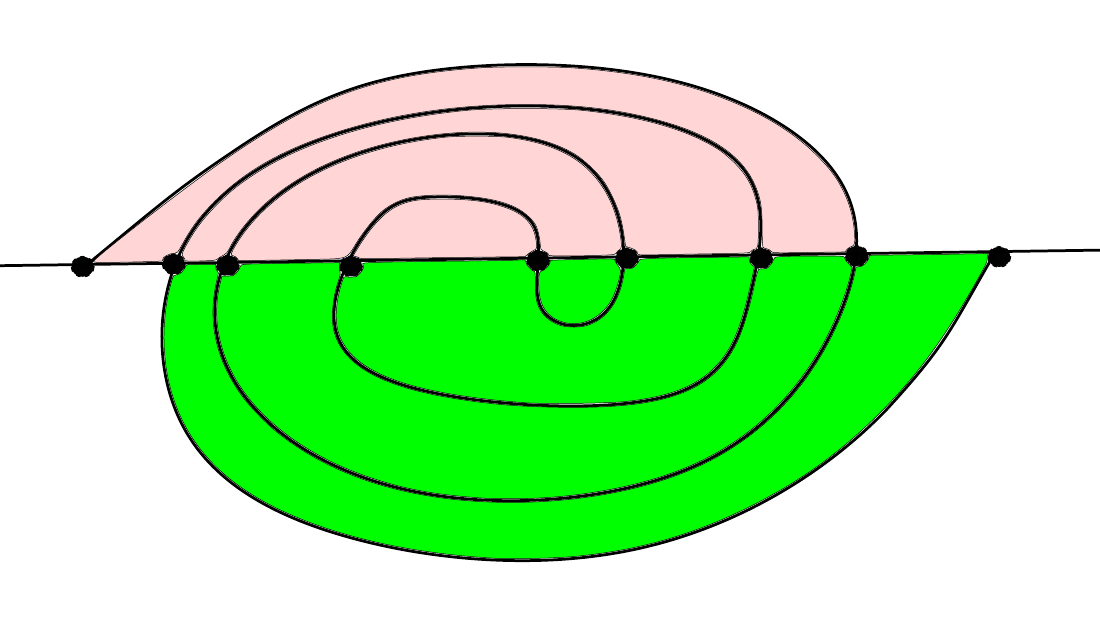}
 \put(100,35){$L_0$}
  \put(35,0){$L_1$}
   \put(35,10){$B$}
    \put(35,40){$A$}
 \end{overpic}
  \caption{The areas $A$ in pink and $B$ in green}\label{subfig-11a}
 \end{subfigure}\hskip 2cm
\begin{subfigure}{0.4\textwidth}
\centering
\begin{overpic}[width=7cm]{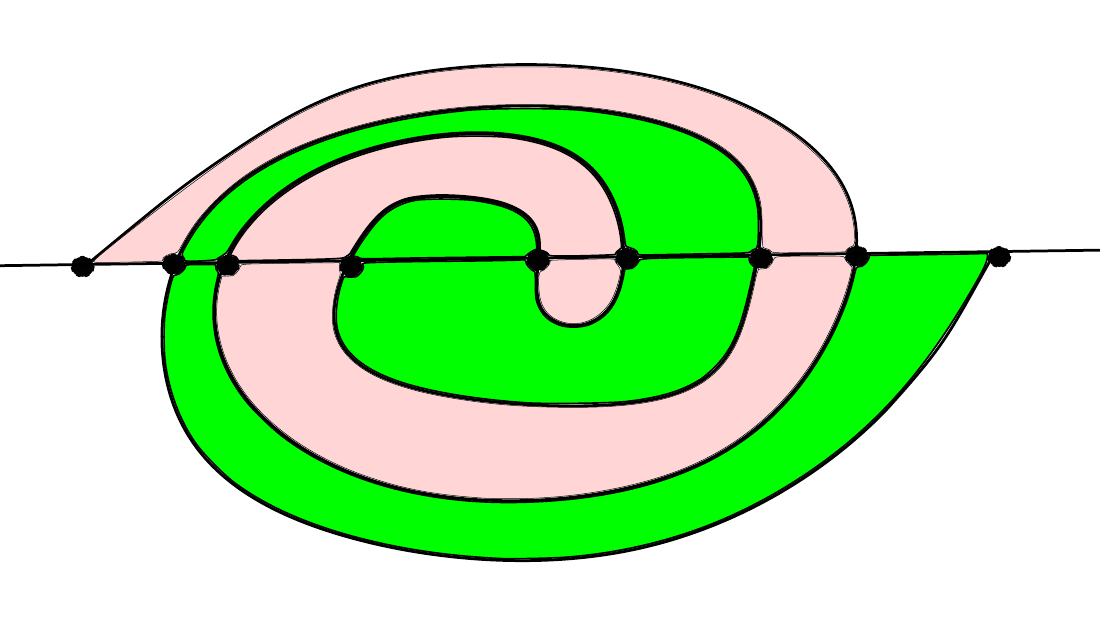}
 \put(35,25){$B'$}
    \put(35,40){$A'$}
 \end{overpic}
 \caption{ The  areas $A'$ in pink and $B'$ in green}\label{subfig-11b}
 \end{subfigure}
 
\caption{ $\gamma(L_0,L_1)$ is bounded below by the smallest of the  areas $A,A',B,B'$(Illustrations from \cite{MCA-VH-CV}).}
\label{Fig-2}
\end{figure}

We now remind the reader of
\begin{lem}[{\cite[Lemma 5.12, p.788]{MCA-VH-CV}}]\label{Lemma-1.2}
  Assume $L_0$ is the zero section and $L_1$ is a smooth spiral as described above. Then, 
\[\gamma(L_0,L_1)\geq \min(A,A',B,B').\]
\end{lem}

\begin{proof}[Proof of Theorem \ref{Thm-non-compactness} in the case $n=1$]
    
We consider a Hamiltonian of the form $H(r,\theta)=h(\frac{r^2}{2})$ the symplectic form being $rdr\wedge d\theta$, so the Hamiltonian vector field (in polar coordinates) being $X_r=0, X_\theta=-h'(\frac{r^2}{2})$ and the flow given by $\Phi^t(r,\theta)=(r,\theta-th'(\frac{r^2}{2}))$. 
So  if $h'\geq 0$, the set $\Phi^t(L_0)$ is  the union of two symmetric spirals as in Figure \ref{Fig-1}. The value of $\gamma(L_0,L_1)$ is bounded from below by the smallest of the areas $A,A',B,B'$. But $L_1$ is symmetric since both the Hamiltonian and $L_0$ are invariant by $(r,\theta)\mapsto (r,\theta+\pi)$ and this implies that $\Phi^t(-x,0)=-\Phi^t(x,0)$. As a result $A'=B'$ and $A=B$. Moreover since it is clear from the picture that $A+B=A'+B'$, we get $A=A'=B=B'$.
Now we choose $h$ such that $h'>0$ on $[0, 1/2)$ and $h'=0$ on $[1/2, \infty)$.
Then, as $t$  goes to $\infty$, the first return point of $L_0$ is closer and closer to $(1,0)$ or $(-1,0)$, so we easily see that 
$\gamma(L_0,\Phi^t(L_0))$ converges to half the area of the disc of radius $1$, that is $\frac{\pi}{2}$. 

In any case since for $t\geq 1$ $\gamma(L_0,\Phi^t(L_0))$ is positive, continuous and has a positive limit at $+\infty$ there is a positive constant $\varepsilon_0$ such that for $t\geq 1$, 
$$\gamma(L_0,\Phi^t(L_0))\geq \varepsilon_0 >0$$ 
As a result, setting $L_k=\Phi^k(L_0)$, we have  for $k< l$ that $$\gamma(L_k,L_l)=\gamma(\Phi^k(L_0), \Phi^l(L_0))=\gamma(L_0, \Phi^{l-k}(L_0))\geq \varepsilon_0>0 $$
\end{proof}
\begin{proof}[Proof of Theorem \ref{Thm-non-compactness} in the general case.]
Let us first prove this in $DT^*([-1,1])\times DT^*\Torus^{n-1}$. 
Consider coordinates $(x_1,p_1)$ in $DT^*([-1,1])$ and $(r,\theta)=(r_2,...,r_n, \theta_2,...,\theta_n)$ in $DT^*\Torus^{n-1}$. We consider the Hamiltonian flow of $\chi(r)H(x_1,p_1)$ where $H$ is the Hamiltonian defining $\Phi^t$ and $\chi$ is supported in $D^{n-1}(1)$, equal to $1$ near $0$:
$$\Psi^t(x_1,p_1,r,\theta)=\left ( \Phi^{t\chi(r)}(x_1,p_1),r, \Theta_t(x_1,p_1,r,\theta)\right )$$
The function $\Theta$ is given by 
$$\Theta_t(x_1,p_1,r,\theta)=\theta+ t\nabla \chi (r)H(x_1,p_1)$$
Now the reduction of $\Lambda_k=\Psi^k(0_\mathbb R\times 0_{\Torus^{n-1}})  \subset DT^*([-1,1])\times DT^*\Torus^{n-1}$ by $r=0$ is given by $\Psi^k(0_{\mathbb R})=L_k$. 
By the reduction inequality (see \cite[Proposition 5.1]{Viterbo-STAGGF}), we have for $k\neq l$
$$ \gamma(\Lambda_k,\Lambda_l) \geq \gamma(L_k,L_l) \geq \varepsilon_0>0 $$
and we see that the space $\widehat{\mathcal L}(DT^*[-1,1]\times DT^*\Torus^{n-1})$ is not totally bounded. 

Finally, let $M$ be any symplectic manifold, and $V_0$ an exact Lagrangian submanifold. We identify a neighbourhood of $V_0$ to $DT^*V_0$ (up to a rescaling of the symplectic form).  Consider a Darboux chart identifying a neighbourhood of a point to $DT^*D^n$. Since $[-1,1]\times \Torus^{n-1}$ embeds in $D^n$, this embedding will extend to a symplectic embedding of $DT^*([-1,1]\times \Torus^{n-1})$ in our Darboux chart. Then we can extend $\Psi^t$ to a Hamiltonian diffeomorphism of $DT^*D^n$ hence of $DT^*V_0$ (and thus of $M$ !).

We thus set $V_k=\Psi^k(V_0)$. We claim that $$\gamma(V_k,V_l)=\gamma(\Lambda_k,\Lambda_l)\geq \gamma(L_k,L_l)\geq \varepsilon_0$$ 
The two right hand side inequalities follow from the reduction inequality and the case $n=1$. The first equality follows by applying the following 
\begin{prop}\label{Prop-locality-of-gamma}
    Let $L$ be an exact Lagrangian in $(M,\omega)$ and $\varphi^t$ a Hamiltonian isotopy supported in a Darboux chart identified to $DT^*U$ where $L=0_U$. Then 
    $\gamma(L,\varphi^1(L))=\gamma(0_U,\varphi^1(0_U))$. 
\end{prop}
We shall need the following Lemma, which is an obvious consequence of  \cite[Proposition D2 and Corollary D3]{Guillermou-Viterbo}:
\begin{lem}\label{Lemma-from-GVit}
    Let $U$ be  a codimension $0$ submanifold with boundary             and $L$ an exact Lagrangian Hamiltonianly isotopic to the zero section $0_U$ in $T^*U$ (so $L=0_U$ near its boundary). Let $W$ be a closed manifold such that $U$ is a codimension $0$ submanifold and let $\Lambda$ be the extension of $L$ by the zero section of $T^*W$ (i.e. $\Lambda=0_W$ outside $T^*U$). Then $\gamma(\Lambda,0_W)$ does not depend on the choice of $W$.  
\end{lem}
\begin{proof}[Proof of Proposition \ref{Prop-locality-of-gamma}]
    We can identify $U$ to an open set with boundary in $L$, and $T^*U$ to a subset of $T^*L$. Now the Floer cohomology of $FH^*(L,\varphi^1(L))$ in $(M,\omega)$ is the same as the Floer cohomology in $T^*L$. Indeed, the boundary of $DT^*L$ in $(M,\omega)$ being of restricted contact type - by exactness of $L$ -, it can be made pseudoconvex by choosing a suitable almost complex structure. As a result and arguing as in the proof of \cite[theorem 4.1]{Viterbo-FCFH1}, we see that the (filtered) Floer cohomology of $(L,\varphi^1(L))$ in $T^*L$ and in $(M,\omega)$ are the same, since holomorphic strips with boundary in $L\cup \varphi^1(L)$ are either contained in $DT^*L$ or have area bounded from below by some constant $\alpha$. For $0<c<1$ we denote by $c\varphi^1(L)$ the image of $\varphi^1(L)$ by $(q,p)\mapsto (q,cp)$ (in $DT^*L$). We then claim that $FH^*(L,\varphi^1(L), a)=FH^*(L,c\varphi^1(L), ca)$. Indeed the barcode decomposition of $FH^*(L,c\varphi^1(L))$ in $(M,\omega)$, say $B_c = \bigoplus_{i\in I} k_{[a_i^c, b_i^c)}$, moves continuously with $c$ and, since $L$ is exact, the ends of the bars belong to the finite set $f_c(L \cap c\varphi^1(L))$, where $f_c$ is a primitive of the restriction to $c\varphi^1(L)$ of the Liouville form of $T^*L$.  We have $f_c = c f_1$ and it follows that the sets of bars are in bijection and $a_i^c = c a_i^1$, $b_i^c = c b_i^1$. This implies $FH^*(L,\varphi^1(L), a)=FH^*(L,c\varphi^1(L), ca)$, as claimed.      
    But for $c$ small enough there can be no holomorphic strip with boundary in $L\cup c\varphi^1(L)$ with area greater than $\alpha_0$ since all intersection points will have action in $[-cA, cA]$ for some $A$. Therefore all holomorphic strips computing the boundary map in Floer cohomology are contained in $DT^*L$, hence the filtered Floer cohomology of $(L,\varphi^1(L))$ in $DT^*L$ coincides with the filtered Floer cohomology in $(M,\omega)$. This implies that $\gamma(L,\varphi^1(L))$ is the same in $DT^*L$ and in $(M,\omega)$. 
   Now applying Lemma \ref{Lemma-from-GVit}, we conclude that $\gamma(L,\varphi^1(L))=\gamma(0_U,\varphi^1(0_U))$. 
\end{proof}
The proof of the first part of our theorem is thus concluded. For the case of $\DHam(M,\omega)$ it follows immediately from the property $\gamma(\varphi,\psi)\geq \gamma(\varphi(L),\psi(L))$.  
\end{proof}

\printbibliography

\end{document}